\documentclass{amsart}

\usepackage{amsfonts, amssymb, amsmath, eucal, verbatim, amsthm, amscd, enumerate}

\usepackage{graphicx}
\usepackage{framed}
\usepackage{tikz}

\newtheorem{theorem}{Theorem}[section]
\newtheorem{lemma}[theorem]{Lemma}
\newtheorem{proposition}[theorem]{Proposition}

\theoremstyle{definition}

\theoremstyle{remark}
\newtheorem*{remark}{Remark}

\newtheorem*{acknowledgements}{Acknowledgements}

\newcommand{\vertiii}[1]{{\left\vert\kern-0.25ex\left\vert\kern-0.25ex\left\vert #1 
    \right\vert\kern-0.25ex\right\vert\kern-0.25ex\right\vert}}

\numberwithin{equation}{section}

\usepackage{setspace}

\def\1{\textbf{\rm 1}}

\begin{document}

\date{\today}

%\author{Jonathan Bennett??}
%\author[Bez]{Neal Bez}
%\address[Neal Bez]{Department of Mathematics, Graduate School of Science and Engineering,
%Saitama University, Saitama 338-8570, Japan}
%\email{nealbez@mail.saitama-u.ac.jp}
%\author[Hong]{Younghun Hong}
%\address[Younghun Hong]{Department of Mathematics, Yonsei University, Seoul 03722, Korea}
%\email{younghun.hong@yonsei.ac.kr}
%\author[Lee]{Sanghyuk Lee}
%\address[Sanghyuk Lee]{Department of Mathematical Sciences, Seoul National University, Seoul 151-747, Korea}
%\email{shklee@snu.ac.kr}
\author[Nakamura]{Shohei Nakamura}
\address[Shohei Nakamura]{Department of Mathematics and Information Sciences, Tokyo Metropolitan University,
1-1 Minami-Ohsawa, Hachioji, Tokyo, 192-0397, Japan}
\email{nakamura-shouhei@ed.tmu.ac.jp}
%\author[Sawano]{Yoshihiro Sawano} 
%\address[Yoshihiro Sawano]{Department of Mathematics and Information Sciences, Tokyo Metropolitan University,
%1-1 Minami-Ohsawa, Hachioji, Tokyo, 192-0397, Japan}
%\email{ysawano@tmu.ac.jp}

\title[The orthonormal Strichartz inequality on torus]{The orthonormal Strichartz inequality on torus}

\begin{abstract}
In this paper, motivated by recent important works due to Frank-Lewin-Lieb-Seiringer \cite{FLLS} and Frank-Sabin \cite{frank-sabin-1}, we study the Strichartz inequality on torus with the orthonormal system input and obtain sharp estimates in certain sense. 
An application of the inequality shows the well-posedness to the periodic Hartree equation describing the infinitely many quantum particles with the power type interaction.    
\end{abstract}

\maketitle

\section{Introduction and Main results}

The classical Strichartz inequality for the free Schr\"{o}dinger propagator $e^{it\Delta}$ may be stated that for any space dimension $d\geq1$ and any admissible pair $p,q\in[1,\infty]$, namely $\frac 2p +\frac dq =d$ and $(p,q,d)\neq (1,\infty,2)$, 
\[ \big\| | e^{it\Delta} f |^2\big \|_{L^p_tL^q_x(\mathbb{R}^{d+1})} \lesssim 1\]
holds as long as $\|f\|_{L^2(\mathbb{R}^d)}=1$ where the notation $\lesssim$ denotes the inequality with some implicit constant, for example, $A\lesssim B$ means an inequality $A\leq CB$ holds for some constant $C>0$.
Such inequality is first observed by Strichartz in \cite{Strichartz} and later extended to mixed norm setting and applied for nonlinear Schr\"{o}dinger equations, for example \cite{GinibreVelo,GuoPeng,KeelTao,Tutumi,Yajima}. 
To explain the problem we address in, let us overview two topics concerning the classical Strichartz inequality, the first one is the generalization of the Strichartz inequality involving the orthonormal system and the second one is the theory for the nonlinear periodic Schr\"{o}dinger equation, especially the Strichartz inequality on torus. 
\subsection{Orthonormal Strichartz inequality on $\mathbb{R}^d$}
Recently, the classical Strichartz inequality has been generalized to the orthonormal setting by Frank-Lewin-Lieb-Seiringer \cite{FLLS} and Frank-Sabin \cite{frank-sabin-1}.  
Let us recall what the orthonormal Strichartz inequality is and their results. For the admissible pair $p,q$ and suitable $\alpha\in[1,\infty]$, we consider the inequality 
\begin{equation}\label{e:ONS-Rd}
\bigg\| \sum_j\lambda_j |e^{it\Delta}f_j|^2 \bigg\|_{L^p_tL^q_x(\mathbb{R}^{d+1})} \lesssim \|\lambda\|_{\ell^\alpha}
\end{equation}
for all $\lambda=(\lambda_j)_j\in\ell^\alpha$ and all orthonormal system $(f_j)_j$ in $L^2(\mathbb{R}^d)$. 
Clearly, the case $\alpha=1$ follows from the triangle inequality and the classical Strichartz inequality without any making use of the orthonormal hypothesis. So, in view of the inclusion relation of $\ell^\alpha$ space, the problem we are interested in is to find the largest $\alpha=\alpha(p,q)$ for which the inequality \eqref{e:ONS-Rd} holds given the admissible pair $p,q$. 
It is convenient to introduce some notations to overview the known results, see Figure \ref{f:points}: 
\[ A= \big(\frac{d-1}{d+1},\frac{d}{d+1}\big),\quad B=(1,0),\quad C=\big(\frac{d-2}{d},1\big). \]
When $d=1$, $A=C=(0,\frac12)$. 
For two points $X,Y\in [0,1]^2$, we use a notation $(X,Y)$ to represent the open line combining $X,Y$. Similarly, we define $[X,Y]$, $(X,Y]$ and $[X,Y)$.  

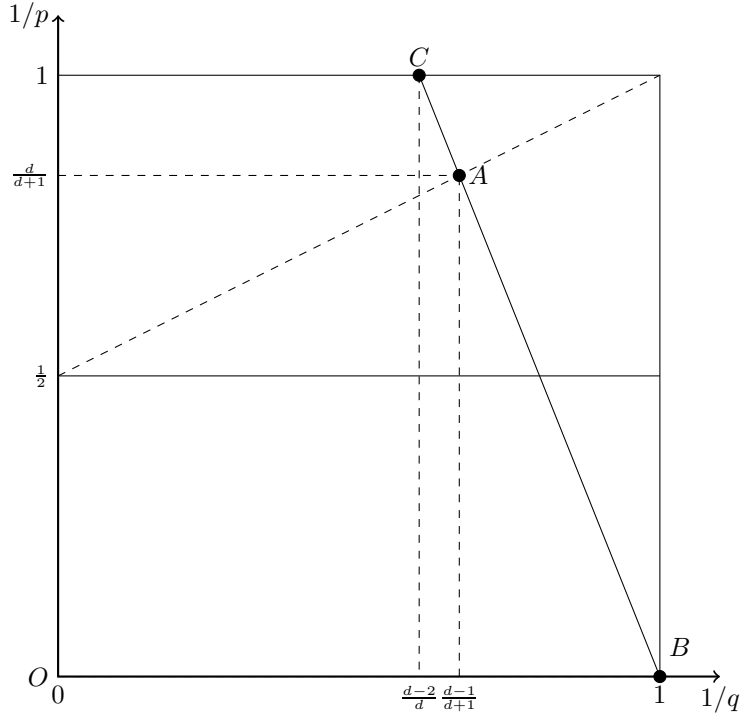
\begin{figure}
\begin{center}

\begin{tikzpicture}[scale=4]

% Draw axes
\draw [<->,thick] (0,2.2) node (yaxis) [left] {$1/p$}
|- (2.2,0) node (xaxis) [below] {$1/q$};
\draw (0, 0) rectangle (2, 2);

\node [left] at (0,0) {$O$};
\node [above] at (1.2,2) {$C$};

\draw (4/3,5/3)--(2,0);

\coordinate (C) at (1.2,2);
\filldraw[fill=black] (C) circle[radius=0.2mm];

\coordinate (B) at (2,0);
\filldraw[fill=black] (B) circle[radius=0.2mm];

\coordinate (Frank-Sabin) at (4/3,5/3);
\filldraw[fill=black] (Frank-Sabin) circle[radius=0.2mm];

\node [right]at (4/3,5/3) {$A$};
\node [right] at (2,.1) {$B$};

\draw [dashed] (0,1)--(2,2);

\draw [dashed] (1.2,2)--(1.2,0); 
\node[below] at (1.2,0) {\tiny $\frac{d-2}{d}$};

% %labels
\node [left]at (0,2) {1};
\node [left]at (0,1) {\tiny $\frac12$};
\node[below] at (0,0) {0};
\node[below] at (2,0) {1};
% % % %

\draw (0,1)--(1,1)--(2,1);

\draw  (4/3,5/3)--(1.2,2);
%%y=-\frac52 x+5

\draw [dashed] (4/3,0)--(4/3,5/3);
\node [below] at (4/3,0) {\tiny $\frac{d-1}{d+1}$};

\draw [dashed] (0,5/3)--(4/3,5/3);
\node [left] at (0,5/3) {\tiny $\frac{d}{d+1}$};

\draw [dashed] (0,0)--(2,0);

\end{tikzpicture}

\end{center}
\caption{The points $A$ to $C$ for $d\geq3$}\label{f:points}

\end{figure}

\begin{theorem}[\cite{FLLS,frank-sabin-1}]\label{t:ONS-Rd}
Let $d\geq1$. If $(\frac1q,\frac1p) \in (A,B]$, then \eqref{e:ONS-Rd} holds for any $\lambda=(\lambda_j)_j\in\ell^{\alpha}$ and any orthonormal system $(f_j)_j$ in $L^2(\mathbb{R}^d)$ whenever $\alpha\leq\frac{2q}{q+1}$. 
Moreover, this is sharp in the sense that the inequality \eqref{e:ONS-Rd} fails if $\alpha>\frac{2q}{q+1}$. 
\end{theorem}

While this theorem gives the answer to the problem on $(A,B]$, namely $\alpha=\frac{2q}{q+1}$ is the best possible, this theorem does not cover all admissible exponents and the problem on $[A,C]$ is still open regardless of recent contributions \cite{BHLNS,FLLS,frank-sabin-2}.  
As far as we are aware, the following are the best known results on $[A,C]$. 
\begin{theorem}[\cite{BHLNS, FLLS, frank-sabin-2}]\label{t:ONS-Rd-beyond}
Let $d\geq1$. 
\begin{enumerate}
\item (Critical point)
On the point $(\frac1q,\frac1p)=A$, the estimate \eqref{e:ONS-Rd} with $\alpha=\frac{2q}{q+1}=p=\frac{d+1}{d}$ fails.
\item
On the region $(\frac1q,\frac1p)\in (A,C)$, the estimate \eqref{e:ONS-Rd} holds as long as $\alpha < p$ and this is sharp up to $\varepsilon$-loss in the sense that \eqref{e:ONS-Rd} fails if $\alpha > p$. 
Moreover, the weak type estimate 
\[ \bigg\| \sum_j \lambda_j |e^{it\Delta}f_j|^2 \bigg\|_{L^{p,\infty}_tL^q_x(\mathbb{R}^{d+1})} \lesssim \|\lambda\|_{\ell^p} \]  
also holds true for any $\lambda=(\lambda_j)_j\in\ell^p$ and any orthonormal system $(f_j)_j$ in $L^2(\mathbb{R}^d)$ where $L^{p,\infty}_t$ is the weak $L^p$-space.     
\item (Keel-Tao endpoint)
On the point $(\frac1q,\frac1p)=C$,  the estimate \eqref{e:ONS-Rd} holds with $\alpha=1$ and this is sharp in the sense that \eqref{e:ONS-Rd} fails if $\alpha>1$.
\end{enumerate}
\end{theorem} 
From this theorem, one may notice that the point $A$ plays a critical role in the sense that the sharp exponent is $\alpha = \frac{2q}{q+1}$ on the lower region and the expected sharp exponent is $\alpha=p$ on the upper region. 

Such generalization involving the orthonormal system is strongly motivated by the theory for the many body quantum mechanics and it is important to find the sharp sequence exponent $\alpha$ as in Theorem \ref{t:ONS-Rd} in this context. The first initiative work of such generalization goes back to the famous work due to Lieb-Thirring \cite{Lieb-Thirring-1} where the Gagliardo-Nirenberg-Sobolev inequality was generalized to the orthonormal inequality, so-called Lieb-Thirring's inequality. Importantly, the sharp orthonormal inequality played a crucial role to prove the stability of matter \cite{LiebBAMS,Lieb-Thirring-1}, see also \cite{sabin-2}. It is also notable that the sharp orthonormal Strichartz inequality as in Theorem \ref{t:ONS-Rd} was employed crucially to establish well-posedness and scattering theory for the certain Hartree equation in \cite{CHP-1,CHP-2,LewinSabinWP,LewinSabinScatt,sabin}. 
  
\subsection{One functional Strichartz inequality on torus}  
There is another theory regarding the classical Strichartz inequality, namely the nonlinear periodic PDE problem. In \cite{Bourgain-restrictiontori} Bourgain studied the nonlinear periodic Schr\"{o}dinger equation on torus $\mathbb{T}^d=(\mathbb{R}/\mathbb{Z})^d$ and established the well-posedness theory. One crucial feature of the equation on $\mathbb{T}^d$ is that the dispersion of the solution is weaker than the solution of the equation on $\mathbb{R}^d$ since $\mathbb{T}^d$ is compact and hence, new difficulty occurs to established the well-posedness theory. %Regardless of such difficulty, he managed to prove the well-posedness result. 
A decisive tool to study the nonlinear periodic Schr\"{o}dinger equation  is the Strichartz inequality on torus which can be stated as follows:
\begin{theorem}[\cite{Bourgain-restrictiontori,BourgainDemeter-decoupling}]
Let $d\geq1$ and $p_* = \frac{d+2}{d}$. Then for arbitrary small $\varepsilon>0$, there exists $C_\varepsilon>0$ such that for any $N>1$ and any $f\in L^2(\mathbb{T}^d)$ whose Fourier support is contained in $[-N,N]^d$, 
\begin{equation}\label{e:strichartz-tori}
\big\| |e^{it\Delta}f|^2 \big\|_{L^{p_*}_{x,t}(\mathbb{T}^{d+1})} \leq C_\varepsilon N^\varepsilon \|f\|_{L^2(\mathbb{T}^d)}
\end{equation}
holds. 
\end{theorem}
Remark that the $N^\varepsilon$-loss in \eqref{e:strichartz-tori} is not removable. 
Historically, in \cite{Bourgain-restrictiontori}, Bourgain proved \eqref{e:strichartz-tori} when $d=1,2$ via number theoretical  argument so-called Hardy-Littlewood circle method and conjectured that \eqref{e:strichartz-tori} holds for any $d\geq3$. 
After some improvements were obtained in \cite{BourgainDemeter-tori-1,BourgainDemeter-tori-2}, this conjecture was finally solved positively by the celebrated work due to Bourgain-Demeter \cite{BourgainDemeter-decoupling} where they employed deep theory from Harmonic analysis so-called decoupling theorem. 
Moreover, it was also observed that the inequality \eqref{e:strichartz-tori} still holds replacing the torus by more general irrational torus. For further discussion and the theory on the irrational torus including survey, see  \cite{CatoireWang,Demirbas,GuoOhWang,KenigPonceVega,Nahmod,Staffilani,Vega}. It is notable that in \cite{BurqGerardTzvetkov}, Burq-G\'{e}rard-Tzvetkov 
studied the nonlinear Schr\"{o}dinger equation on the compact manifold. In this paper, we employ their idea used to establish the Strichartz inequality on the compact manifold.  
Further improvement were obtained in their continued works \cite{BurqGerardTzvetkov-2,BurqGerardTzvetkov-3} where they employed bilinear and multilinear approach. 
For the study of the Hartree equation on compact manifold, see the work of G\'{e}rard-Pierfelice \cite{GerardPierfelice}.

\subsection{Main results}
With these two topics concerning the classical Strichartz inequality in mind, it is natural to investigate the nonlinear periodic equation in the framework of orthonormal systems. So, our main aim in this paper is to establish the sharp orthonormal Strichartz inequality on torus and apply it to the periodic Hartree equation for the density matrices of infinite trace. More precisely, our first main goal is to determine the largest $\alpha$ for which the inequality 
\begin{equation}\label{e:ONS-Td}
\bigg\| \sum_{j} \lambda_j |e^{it\Delta} P_{\leq N} f_j|^2 \bigg\|_{L_{t}^{p}L^q_x(\mathbb{T}^{d+1})} \leq C_\rho N^\rho \|\lambda\|_{\ell^\alpha}                                                                                                                                 
\end{equation}
holds for any $N>1$, any $\lambda=(\lambda_j)_j\in\ell^\alpha$ and any orthonormal system $(f_j)_j$ in $L^2(\mathbb{T}^d)$, given a parameter $\rho>0$ and admissible pair $p,q$. Here, the operator $P_{\leq N}$ denotes the frequency cut-off operator which is defined by $P_{\leq N}\phi= ( \1_{[-N,N]^d} \hat{\phi} )^{\vee}$, where $(\hat{\phi}(n))_n$ is the Fourier coefficient of $\phi$ and ${}^\vee$ is its inverse.  
When $p=q=p_*$, again applying the triangle inequality and \eqref{e:strichartz-tori}, we can prove for any small $\varepsilon$, 
\begin{equation}\label{e:ONS-one}
\bigg\| \sum_{j} \lambda_j |e^{it\Delta} P_{\leq N} f_j|^2 \bigg\|_{L_{x,t}^{p_*}(\mathbb{T}^{d+1})} \leq C_\varepsilon N^\varepsilon \|\lambda\|_{\ell^\1}.                                                                             
\end{equation}
Our first observation is that if we define $\alpha(\rho)$ for each $\rho>0$ by  
\begin{equation}\label{e:alpha}
\frac{1}{\alpha(\rho)} = 1 - \frac\rho d, 
\end{equation}
then $\alpha\leq \alpha(\rho)$ is necessary for the inequality \eqref{e:ONS-Td}, we will see this in Lemma \ref{l:necessary}  by testing the inequality \eqref{e:ONS-Td} with a simple example. 
So, in the orthonormal framework, the sharp exponent $\alpha$ for the inequality \eqref{e:ONS-Td} should be related to the power $\rho$ and more interestingly, we can easily see that $\alpha(\rho)\to1$ as $\rho\to0$. 
This reveals that the trivial estimate \eqref{e:ONS-one} is almost sharp when $\varepsilon\to0$. In other words, to make $\alpha$ strictly bigger than one, we need to lose the factor $N$ with certain power. Our first result is the following. 

\begin{theorem}\label{t:ONS-pure}
Let $d\geq1$ and $\rho\in(0,\frac{1}{p_*}]$. Then 
\begin{equation}\label{e:pure-ONS}
\bigg\| \sum_{j} \lambda_j |e^{it\Delta} P_{\leq N} f_j|^2 \bigg\|_{L_{t,x}^{p_*}(\mathbb{T}^{d+1})} \leq C_\rho N^\rho \|\lambda\|_{\ell^\alpha}                                                                   
\end{equation}
holds for any $N>1$, any $\lambda\in\ell^{\alpha}$ and any orthonormal system $(f_j)_j$ in $L^2(\mathbb{T}^d)$ whenever $\alpha < \alpha(\rho)$. Moreover, this is sharp up to $\varepsilon$-loss in the sense that \eqref{e:pure-ONS} fails if $\alpha>\alpha(\rho)$.
\end{theorem}

Remark that the possibility of \eqref{e:pure-ONS} with the expected exponent $\alpha=\alpha(\rho)$ remains open except the case $\rho=\frac1p$. 
Theorem \ref{t:ONS-pure} is a consequence of the following more general mixed norm orthonormal Strichartz inequality via the complex interpolation with \eqref{e:ONS-one}.
Note that one can easily check that $\alpha(1/p)=\frac{2q}{q+1}$ holds if $\frac2p+\frac dq=d$.  

\begin{theorem}\label{t:ONS-mix}
Let $d\geq1$ and $(\frac1q,\frac1p) \in (A,B]$. Then for any $N>1$, any $\lambda\in\ell^{\alpha}$ and any orthonormal system $(f_j)_j$ in $L^2(\mathbb{T}^d)$, 
\begin{equation}\label{e:ONS-mix}
\bigg\| \sum_{j} \lambda_j |e^{it\Delta} P_{\leq N} f_j|^2 \bigg\|_{L_{t}^{p}L^q_x(\mathbb{T}^{d+1})} \leq C N^\frac1p \|\lambda\|_{\ell^{\alpha}}                                                                                                                                 
\end{equation}
holds true whenever $\alpha \leq \frac{2q}{q+1}$. Moreover, this is sharp in the sense that \eqref{e:ONS-mix} fails if $\alpha > \frac{2q}{q+1}$.
\end{theorem}

Recall that the exponent $\alpha(1/p)=\frac{2q}{q+1}$ has already appeared in Theorem \ref{t:ONS-Rd} as the sharp exponent for the orthonormal Strichartz inequality on $\mathbb{R}^d$.  Furthermore, the range $(A,B]$ also corresponds to the range of Theorem \ref{t:ONS-Rd}. So,  we may find some connections between the orthonormal Strichartz inequality on $\mathbb{R}^d$ and the one on $\mathbb{T}^d$ with the case  $\rho=\frac1p$. 
It is natural to ask further what happens in the region $[A,C]$. In view of the similarity between the $\mathbb{R}^d$ case and the $\mathbb{T}^d$ case with $\rho=\frac1p$ and Theorem \ref{t:ONS-Rd-beyond}, one may expect some different phenomena on $[A,C]$. Especially, recall that at the point $(\frac1q,\frac1p)=A$, the inequality on $\mathbb{R}^d$:
\[
\bigg\| \sum_{j} \lambda_j |e^{it\Delta} f_j|^2 \bigg\|_{L^p_tL^q_x(\mathbb{R}^{d+1})} \lesssim \|\lambda\|_{\ell^{\frac{2q}{q+1}}} 
\]  
fails. 
In spite of such similarity and the failure, we interestingly have a positive result at the point  $A$ for $\mathbb{T}^d$ case at least when $d=1$.
Recall that when $d=1$, exponents are $A=C=(0,\frac12)$ and $\alpha(1/p)=\frac{2q}{q+1}=2$.

\begin{theorem}\label{t:endpoint-1d}
Let $(\frac1q,\frac1p)=A=(0,\frac12)$. Then for any $N>1$, $\lambda\in\ell^\alpha$ and any orthonormal system $(f_j)_j$ in $L^2(\mathbb{T})$, 
\begin{equation}\label{e:ONS-d1}
\bigg\| \sum_j \lambda_j |e^{it\Delta}P_{\leq N}f_j|^2  \bigg\|_{L^2_tL^\infty_x(\mathbb{T}^{1+1})} \leq C N^{\frac12} \|\lambda\|_{\ell^{\alpha}}  
\end{equation}
holds true whenever $\alpha\leq 2$. Moreover, this is sharp in the sense that \eqref{e:ONS-d1} fails if $\alpha>2$. 
\end{theorem}  

%We combine ideas of Burq-G\'{e}rard-Tzvetkov \cite{BurqGerardTzvetkov}, Vega \cite{Vega} and the $TT^*$ argument in Schatten space due to Frank-Sabin \cite{frank-sabin-1} to prove Theorem \ref{t:ONS-mix}. In particular, we do not make use of any number theoretical arguments for the proof of Theorem \ref{t:ONS-mix} such as the Hardy-Littlewood circle method which is clucially used to show \eqref{e:strichartz-tori} by Bourgain \cite{Bourgain-restrictiontori}. Meanwhile, for the proof of Theorem \ref{t:endpoint-1d}, we follow the spilt of the Hardy-Littlewood circle method via Frank-Sabin's $TT^*$ argument in Schatten space. 
We emphasize that to prove the endpoint estimate Theorem \ref{t:endpoint-1d} we follow the spirit of the Hardy-Littlewood circle method via Frank-Sabin's $TT^*$ argument in Schatten space. This is possible since the right-hand side of \eqref{e:ONS-d1} becomes $\ell^2$ when $d=1$ and $(\frac1q,\frac1p) =A= (0,\frac12)$. We will make use of the speciality of $\ell^2$.  The problem on the region $[A,C]$ for $d\geq2$ remains open although we will give one observation in Theorem \ref{t:ONS-beyond}.  
There are some possibility to extend Theorem \ref{t:ONS-mix} to more general compact manifold as Burq-G\'{e}rard-Tzvetkov did from view the point of our proof of Theorem \ref{t:ONS-mix}.  However, we will not go to such direction here. 

As an application of the above orthonormal Strichartz inequalities, we consider $M$ couple of nonlinear periodic Hartree equations which describes the dynamics of $M$ fermions interacting via a power type potential $w_a(x)=|x|^{-a}$ for certain $0<a<d$
\begin{equation}\label{e:Hartree-0}
  \left\{
  \begin{array}{ll}
	i\partial_tu_1&=(-\Delta+w_{a}\ast\rho)u_1,  \quad u_1|_{t=0} = f_1  \\
	&\ \vdots \\ 
	i\partial_tu_M&=(-\Delta+w_{a}\ast\rho)u_M, \quad u_M|_{t=0} = f_M ,  
  \end{array} %\right.
  \right.
\end{equation}
where $(x,t) \in \mathbb{T}^d\times\mathbb{R}$, $(f_j)_{j=1}^M$ is an orthonormal system in $L^2(\mathbb{T}^d)$ and $\rho$ is a density function defined by $\rho(x,t) = \sum_{j=1}^M |u_j(x,t)|^2 $. Remark that the solution $(u_j(t))_{j=1}^M$ continues to be an orthonormal system in $L^2(\mathbb{T}^d)$ for each $t>0$. 
Our main interest is the case $M\to\infty$ and hence, we naturally arrive at the operator valued equivalent formulation of \eqref{e:Hartree-0} as follows: 
\begin{equation}\label{e:Hartree}
  \left\{
  \begin{array}{ll}
	i\partial_t\gamma=[-\Delta+w_{a}\ast\rho_\gamma,\gamma],  \quad (x,t)\in \mathbb{T}^d \times \mathbb{R} \\
	\gamma|_{t=0}=\gamma_0.
  \end{array} %\right.
  \right.
\end{equation}
Here $\gamma_0,\gamma=\gamma(t)$ are bounded and self-adjoint operators on $L^2(\mathbb{T}^d)$, $[A,B]$ is a commutator of two operators $A$ and $B$ and $\rho_\gamma:\mathbb{T}^d\to\mathbb{C}$ is given by $\rho_\gamma(x) = \gamma(x,x)$ where $\gamma(\cdot,\cdot)$ denotes the integral kernel of the operator $\gamma$. 
There are several context for this equation on $\mathbb{R}^d$ when $\gamma_0$ is in the trace class \cite{BovePratoFano1,BovePratoFano2,Chadam} and more importantly Lewin-Sabin \cite{LewinSabinWP,LewinSabinScatt} and Chen-Hong-Pavlovi\'{c} \cite{CHP-1,CHP-2} study the equation \eqref{e:Hartree} when $\gamma$ is not in the trace class. 
We will obtain the $\mathbb{T}^d$ counterpart of the (local) well-posedness result due to Frank-Sabin \cite[Theorem 14]{frank-sabin-1}. To state our result concerning to the equation \eqref{e:Hartree}, let us introduce more notions. 
For $\alpha\in[1,\infty)$, $\mathcal{C}^\alpha=\mathcal{C}^\alpha(L^2(\mathbb{T}^d))$ denotes the Schatten space based on $L^2(\mathbb{T}^d)$ which is the space of all compact operators $A$ on $L^2(\mathbb{T}^d)$ such that ${\rm Tr}|A|^\alpha<\infty$, where $|A|=\sqrt{A^*A}$, and its norm is defined by $\|A\|_{\mathcal{C}^\alpha}=({\rm Tr}|A|^\alpha)^\frac1\alpha$. If $\alpha=\infty$, we define $\|A\|_{\mathcal{C}^\infty} = \|A\|_{L^2\to L^2}$. Also, we use the Sobolev type Schatten space $\mathcal{C}^{\alpha,s}=\mathcal{C}^{\alpha,s}(L^2(\mathbb{T}^d))$, $s\in\mathbb{R}$, introduced in \cite{CHP-1,CHP-2} whose norm is defined by 
\[ \|\gamma\|_{\mathcal{C}^{\alpha,s}(L^2(\mathbb{T}^d))} = \| \langle D\rangle^s \gamma \langle D \rangle^s \|_{\mathcal{C}^\alpha(L^2(\mathbb{T}^d))}, \] 
where $\langle D\rangle^s$ is the inhomogeneous derivative, $\langle D\rangle^s\phi = ((1+|n|^2)^\frac s2 \hat{\phi})^\vee$. 

\begin{theorem}\label{t:wellposed}
%Let $d\geq1$. 
%\begin{enumerate}
%\item
Let $d\geq1$. Suppose $(\frac1q,\frac1p)\in (A,B)$, $\frac{1}{2p}<s$ and $0<a<\frac{3}{2p}$. 
\begin{enumerate}
\item(Local well-posedness)
For any $\gamma_0\in\mathcal{C}^{\frac{2q}{q+1},s}(L^2(\mathbb{T}^d))$, there exist \  $T=T( \|\gamma_0\|_{\mathcal{C}^{\frac{2q}{q+1},s}(L^2(\mathbb{T}^d))} ,s,a) > 0$ and $\gamma \in C^0_t([0,T]; \mathcal{C}^{\frac{2q}{q+1},s}(L^2(\mathbb{T}^d)))$ satisfying \eqref{e:Hartree} on $[0,T]\times \mathbb{T}^d$ and $\rho_\gamma\in L^p_tL^q_x([0,T]\times\mathbb{T}^d)$. 
%\item
%Suppose $d\geq3$, $p=q=p_*$, $s>\frac{d-2}{d+1}$ and $1\leq \alpha<\alpha(\frac{d-2}{d+2})$. Then for any $\gamma_0\in\mathcal{C}^{\alpha}(H^s(\mathbb{T}^d))$ such that $R:=\|\gamma_0\|_{\mathcal{C}^{\alpha}(H^s(\mathbb{T}^d))}<\infty$, there exists $T=T(R)$ and $\gamma \in C^0_t([0,T]; \mathcal{C}^{\alpha}(H^{s}(\mathbb{T}^d)))$ satisfying \eqref{e:Hartree} on $[0,T]\times \mathbb{T}^d$ with $a=1$ and $\rho_\gamma\in L_{t,x}^{p_*}([0,T]\times\mathbb{T}^d)$.
%\end{enumerate}
\item(Almost global well-posedness)
For each $T>0$, we have small $R_T=R_T(a,s)>0$ such that if $\|\gamma_0\|_{\mathcal{C}^{\frac{2q}{q+1},s}(L^2(\mathbb{T}^d))} \leq R_T$, then there exists a solution $\gamma \in C^0_t([0,T]; \mathcal{C}^{\frac{2q}{q+1},s}(L^2(\mathbb{T}^d)))$ satisfying \eqref{e:Hartree} on $[0,T]\times \mathbb{T}^d$ and $\rho_\gamma\in L^p_tL^q_x([0,T]\times\mathbb{T}^d)$. 
\end{enumerate}
\end{theorem} 

Note that if $d=3$ and $(\frac1q,\frac1p)\in(A,B)$ is sufficiently close to $A$, we may choose $a=1$ which is the most meaningful case from view point of physical motivation in Theorem \ref{t:wellposed}. In fact, the condition $\frac32 \cdot \frac{d}{d+1} = \frac32\cdot\frac{3}{3+1}>1$ holds and hence $\frac{3}{2p}>1$ holds if $\frac1p$ is sufficiently close to $\frac{d}{d+1} = \frac{3}{3+1}$ which means $(\frac1q,\frac1p)$ is sufficiently close to $A$, recall $A=(\frac{d-1}{d+1},\frac{d}{d+1})$. 
So, this exhibits one importance of extending the orthonormal Strichartz inequality up to near the point $A$. To have more range of $a$, we need to establish the orthonormal Strichartz inequality on the beyond region $[A,C]$ as in Theorems \ref{t:endpoint-1d} and \ref{t:ONS-beyond}.  
Also, in such case, namely $(\frac1q,\frac1p)$ close to $A$, the gain of the Schatten exponent $\alpha=\frac{2q}{q+1}$ is close to $\frac{d+1}{d}$ which is the largest number among $\{\frac{2q}{q+1}:(\frac1q,\frac1p)\in[A,B]\}$. %While Theorem \ref{t:wellposed} has such positive perspective, we do not take care so much about the derivative loss in Theorem \ref{t:wellposed} where we lose almost $\frac{1}{2p}$-derivative. 

This paper is organized as follows. In Section 2, we give a few definitions and recall the duality principle. In Section 3, we prove orthonormal Strichartz inequality Theorems \ref{t:ONS-pure}, \ref{t:ONS-mix} and \ref{t:endpoint-1d}. In Section 4, we prove the well-posedness result, Theorem \ref{t:wellposed}. In Section 5, we give one observation concerning to the orthonormal Strichartz inequality on the beyond region $[A,C]$ where we will show the almost sharp inequality at $A$ even when $d\geq2$. 

\section{Preliminaries}
In this section, we provide further definitions and recall the duality principle due to Frank-Sabin \cite{frank-sabin-1}. 
For $s\in\mathbb{R}$ and $p\in[1,\infty]$, we use $B^s_{p,\infty}=B^s_{p,\infty}(\mathbb{T}^d)$ to denote the Besov space on $\mathbb{T}^d$ whose norm is defined by 
\[
\|f\|_{B^s_{p,\infty}(\mathbb{T}^d)} = \sup_{k\in \mathbb{N}\cup\{0\}} 2^{ks} \|P_k f\|_{L^p(\mathbb{T}^d)}. 
\] 
Here, $P_k$ is the frequency cutoff operator, $P_k\phi(x)=(\varphi_k\hat{\phi})^\vee$ for $k\in \mathbb{N}\cup\{0\}$ where $\{\varphi_k\}_{k=0}^\infty$ is the partition of unity, namely $\varphi_k$ is a smooth function whose support is contained in $\{|\xi|\sim 2^k\}$ when $k\geq1$ and $\varphi_0$ is a smooth function whose support is contained in $\{|\xi|\leq2\}$ such that $\sum_{k=0}^\infty \varphi_k = 1$. See \cite{Triebel} for the details of this function space. 
It is notable that for $a\in(0,d)$, $w_a(x)=|x|^{-a} \in B^s_{p,\infty}(\mathbb{T}^d)$ if and only if $a \leq \frac dp - s$ holds. We will use this to show Theorem \ref{t:wellposed} in Section 4. In the sequel, we sometimes abbreviate $\mathbb{T}^d$ and use $L^2$ instead of $L^2(\mathbb{T}^d)$ for example.  
It is reasonable to reformulate the inequality \eqref{e:ONS-Td} in terms of the Fourier extension operator. 
Let us introduce the notation $S_{d,N}=\mathbb{Z}^d\cap [-N,N]^d$ and define the Fourier extension operator $\mathcal{E}_N$ by 
\[
\mathcal{E}_N a(x,t)=\sum_{n\in S_{d,N}}a_ne^{2\pi i(x\cdot n+t|n|^2)},\quad (x,t)\in\mathbb{T}^{d+1},
\]
for $a=(a_n)_n\in\ell^2$. 
Then its dual operator $\mathcal{E}_N^*$ (Fourier restriction operator) is given by 
\[
\mathcal{E}_N^*F(n)=
%  \left\{
%  \begin{array}{ll}
	\int_{\mathbb{T}^{d+1}}F(x,t)e^{-2\pi i(x\cdot n+t|n|^2)}\, \mathrm{d}x\mathrm{d}t  %& (n\in S_{d,N}),  \\
%	0 & (n\in\mathbb{Z}^d\setminus S_{d,N}).
%  \end{array} %\right.
%  \right.
\]
if $n\in S_{d,N}$ and $\mathcal{E}_N^*F(n)=0$ if $n\notin S_{d,N}$. 
Here, the dual operator of $\mathcal{E}_N$ means that for any $a\in\ell^2$ and any $F\in L^2(\mathbb{T}^{d+1})$,  
\[
\langle \mathcal{E}_Na,F\rangle_{L^2_{x,t}(\mathbb{T}^{d+1})}=\langle a,\mathcal{E}_N^*F\rangle_{\ell^2_n}
\]
holds. Also, it is notable that from few calculations the operator $\mathcal{E}_N\mathcal{E}_N^*$ is given by 
\begin{align*}
\mathcal{E}_N\mathcal{E}_N^*F(x,t)
&=
\int_{\mathbb{T}}e^{i(t-t')\Delta}[F(\cdot,t')](x)\, \mathrm{d}t'\\
&=
\int_{\mathbb{T}}\sum_{n\in S_{d,N}}\widehat{F(\cdot,t')}(n)e^{2\pi i(x\cdot n+(t-t')|n|^2)}\, \mathrm{d}t',
\end{align*}
and hence if we write 
\[
K_N(x,t)=\sum_{n\in S_{d,N}}e^{2\pi i(x\cdot n+t|n|^2)},
\]
then we have 
\begin{equation}\label{e:convolution}
\mathcal{E}_N\mathcal{E}_N^*F(x,t)
=K_N\ast F(x,t)=
\int_{\mathbb{T}^{d+1}}K_N(x-x',t-t')F(x',t')\, \mathrm{d}x'\mathrm{d}t'.
\end{equation}
Using these notations, the inequality \eqref{e:ONS-Td} can be reformulated as follows. The inequality \eqref{e:ONS-Td} holds for any $N>1$, any $\lambda\in\ell^\alpha$ and any orthonormal system $(f_j)_j$ in $L^2(\mathbb{T}^d)$ if and only if 
\begin{equation}\label{e:ONS-extension}
\bigg\| \sum_j\lambda_j|\mathcal{E}_N a_j|^2 \bigg\|_{L^p_tL^q_x(\mathbb{T}^{d+1})} \leq C_{\rho} N^\rho \|\lambda\|_{\ell^\alpha} 
\end{equation} 
holds for any $N>1$, $\lambda\in \ell^\alpha$ and any orthonormal system $(a_j)_j$ in $\ell^2$. 
This is because if we let $a_j = \hat{f_j}$, then the orthonormality of $(f_j)_j$ in $L^2(\mathbb{T}^d)$ is equivalent to the one of $(a_j)_j$ in $\ell^2$ and $e^{it\Delta}f_j = \mathcal{E}_N a_j$.  
From now on, we will mainly consider the inequality of the form \eqref{e:ONS-extension}. 

All our results concerning to the orthonormal inequality would be shown in terms of the Schatten spaces. In fact, thanks to the duality principle due to Frank-Sabin \cite{frank-sabin-1}, the orthonormal inequality we will prove can be rephrased as follows. 
\begin{lemma}[Lemma 3 in \cite{frank-sabin-1}]\label{l:duality}
The inequality \eqref{e:ONS-extension} is equivalent to 
\begin{equation}\label{e:FS-dual}
\big\| W_1\mathcal{E}_N\mathcal{E}_N^* W_2 \big\|_{\mathcal{C}^{\alpha'}(L^2(\mathbb{T}^{d+1}))} \leq C_\rho N^\rho \|W_1\|_{L^{2p'}_tL^{2q'}_x(\mathbb{T}^{d+1})} \|W_2\|_{L^{2p'}_tL^{2q'}_x(\mathbb{T}^{d+1})}
\end{equation}
for all $W_1,W_2 \in L^{2p'}_tL^{2q'}_x(\mathbb{T}^{d+1})$.
\end{lemma}

\section{Proof of Theorems \ref{t:ONS-pure}, \ref{t:ONS-mix} and \ref{t:endpoint-1d}}
\subsection{The necessity of $\alpha\leq \alpha(\rho)$}
First, we prove the necessity $\alpha\leq\alpha(\rho)$ for the inequality \eqref{e:ONS-Td} by testing a simple example. 
\begin{lemma}\label{l:necessary}
Let $d\geq1$ and $p,q,\alpha\in[1,\infty]$ be arbitrary. 
Suppose \eqref{e:ONS-Td} or equivalently \eqref{e:ONS-extension} with some $\rho>0$ holds for any $N>1$, any $\lambda\in\ell^\alpha$ and any orthonormal system $(a_j)_j$ in $\ell^2$. 
Then it must be $\alpha\leq\alpha(\rho)$. 
\end{lemma}

\begin{proof}
Let $a_j=\1_{\{j\}}$ for each $j\in\mathbb{Z}^d$ and $\lambda_j=\1_{S_{d,N}}(j)$. 
Notice that if $j\in S_{d,N}$, then 
\[| \mathcal{E}_N a_j(x) | = \Big| \sum_{n\in S_{d,N}} e^{2\pi i(x\cdot n + t|n|^2)} \1_{\{j\}}(n) \Big| = 1, \]
which implies 
\[ \bigg\| \sum_j \lambda_j |\mathcal{E}_N a_j|^2 \bigg\|_{ L^p_tL^q_x(\mathbb{T}^{d+1}) } = \sharp S_{d,N} \sim N^d. \]
On the other hand, the right-hand side of \eqref{e:ONS-extension} is 
\[ N^\rho \|\lambda\|_{\ell^\alpha} = N^\rho ( \sharp S_{d,N} )^\frac1\alpha \sim N^\rho N^\frac d\alpha. \]
So, applying \eqref{e:ONS-extension} reveals $ N^d \lesssim N^\rho N^\frac d\alpha$, which gives $d \leq \rho + \frac d\alpha$ as $N\to\infty$.
\end{proof}

As we mentioned in Section 1, $\alpha(\rho) = \frac{2q}{q+1}$ when $\rho=\frac1p$ and $\frac2p+\frac dq=d$. Hence, Lemma \ref{l:necessary} shows the sharpness part of Theorems \ref{t:ONS-mix} and \ref{t:endpoint-1d}.

\subsection{Proof of Theorems \ref{t:ONS-pure} and \ref{t:ONS-mix}}
Let us prove Theorem \ref{t:ONS-mix}. Once we prove Theorem \ref{t:ONS-mix}, then Theorem \ref{t:ONS-pure} follows from the complex interpolation between Theorem \ref{t:ONS-mix} and \eqref{e:ONS-one}. 
In this subsection, we use the notation $I_N=[-\frac{1}{2N},\frac{1}{2N}]$. The key point is the following dispersive estimate observed in Kenig-Ponce-Vega \cite{KenigPonceVega}. 
\begin{lemma}[(5.9) in \cite{KenigPonceVega}]\label{l:dispersive}
It holds that 
\[
\bigg|\sum_{n=-N}^Ne^{2\pi i(xn+t|n|^2)}\bigg|\leq C|t|^{-\frac12}
\]
for any $(x,t)\in  \mathbb{T} \times [-N^{-1},N^{-1}]$. 
\end{lemma} 
From Lemma \ref{l:dispersive}, we clearly have 
\begin{equation}\label{e:dispersive}
\bigg|\sum_{n\in S_{d,N}}e^{2\pi i(x\cdot n+t|n|^2)}\bigg|\leq C|t|^{-\frac d2}
\end{equation}
for any $(x,t)\in\mathbb{T}^d\times [-N^{-1},N^{-1}]$. 
Using this with Stein's analytic interpolation, we prove the following proposition. 
See Vega \cite{Vega} for the one functional counterpart. 
\begin{proposition}\label{p:ONS-Td}
Let $d\geq1$ and suppose $(\frac1q,\frac1p) \in (A,B]$. Then for any $N>1$, any $\lambda\in\ell^\frac{2q}{q+1}$ and any orthonormal system $(a_j)_j$ in $\ell^2$, 
\begin{equation}\label{e:ShortTime}
\bigg\| \sum_j \lambda_j |\mathcal{E}_Na_j|^2\bigg\|_{L^p_tL^q_x(\mathbb{T}^d\times I_N)}
\leq C\|\lambda\|_{\ell^\frac{2q}{q+1}}.
\end{equation}
\end{proposition}
\begin{proof}
Thanks to the duality principle, Lemma \ref{l:duality}, to prove the desired estimate \eqref{e:ShortTime} for all $(\frac1q,\frac1p)\in(A,B]$, it suffices to to show  
\begin{equation}\label{e:dual-ShortTime}
\big\| W_1 \1_{I_N} \mathcal{E}_N \mathcal{E}_N^*[ \1_{I_N} W_2 ] \big\|_{\mathcal{C}^\alpha(L^2(\mathbb{T}^{d+1}))} \lesssim \|W_1\|_{L^\beta_tL^\alpha_x(\mathbb{T}^{d+1})} \|W_2\|_{L^\beta_tL^\alpha_x(\mathbb{T}^{d+1})}
\end{equation}
for all $\alpha,\beta\geq1$ such that $\frac2\beta + \frac d\alpha = 1$ and $0 \leq \frac1\alpha < \frac{1}{d+1}$. 
Moreover, it is enough to show \eqref{e:dual-ShortTime} on $\frac{1}{d+2}\leq \frac1\alpha <\frac{1}{d+1}$ since we trivially have \eqref{e:dual-ShortTime} when $\alpha=\infty$ from the Plancherel theorem. 

Define for $\varepsilon>0$, 
$T_{N,\varepsilon}=K_{N,\varepsilon}\ast$ where $K_{N,\varepsilon}(x,t)=\1_{\varepsilon<|t|<N^{-1}}K_N(x,t)$. 
Once we have 
\begin{equation}\label{0927-1}
\big\|W_1 \1_{I_N} T_{N,\varepsilon}[ \1_{I_N}W_2] \big\|_{\mathcal{C}^\alpha(L^2(\mathbb{T}^{d+1}))}
\leq
C \|W_1\|_{L^\beta_tL^\alpha_x(\mathbb{T}^{d+1})}\|W_2\|_{L^\beta_tL^\alpha_x(\mathbb{T}^{d+1})}
\end{equation}
for some $C$ independent of $\varepsilon$, then \eqref{e:dual-ShortTime} follows by taking $\varepsilon\to0$.
To do Stein's analytic complex interpolation, we further define for $z\in\mathbb{C}$ with ${\rm Re}z\in[-1,\frac{d}{2}]$, 
\[
K_{N,\varepsilon}^z(x,t)
=t^zK_{N,\varepsilon}(x,t) 
\]
and $T_{N,\varepsilon}^z = K^z_{N,\varepsilon}\ast$.
From \eqref{e:dispersive}, we have for $(x,t)\in \mathbb{T}^d\times I_N$ 
\[
|K_{N,\varepsilon}^z(x,t)|
\leq
C
|t|^{{\rm Re}z-\frac d2}.
\]
This involving the Hardy-Littlewood-Sobolev inequality reveals that 
\begin{align*}
&\big\| W_1 \1_{I_N} T_{N,\varepsilon}^z[ \1_{I_N} W_2] \big\|_{\mathcal{C}^2(L^2(\mathbb{T}^{d+1}))}^2\\
=&
\int_{(x,t)\in  \mathbb{T}^d \times I_N} 
\int_{(x',t')\in \mathbb{T}^d \times I_N} 
|W_1(x,t)K_{N,\varepsilon}^z(x-x',t-t')W_2(x',t')|^2\, \mathrm{d}x\mathrm{d}t\mathrm{d}x'\mathrm{d}t'\\
%\leq& 
%C
%\int_{t,s\in I_N}
%\|W_1(t)\|_{L^2_x(\mathbb{T}^d)}^2|t-s|^{2{\rm Re}z-d}\|W_2(s)\|_{L^2_x(\mathbb{T}^d)}^2dtds\\
\leq&
C
\big\| \|W_1\|_{L^2_x(\mathbb{T}^d)}^2\big\|_{L^{\tilde{u}}_t(\mathbb{T})}
\big\| \|W_2\|_{L^2_x(\mathbb{T}^d)}^2\big\|_{L^{\tilde{u}}_t(\mathbb{T})},
\end{align*}
where 
$2{\rm Re}z-d\in(-1,0]$ and $\frac{2}{\tilde{u}}+( d-2{\rm Re}z)=2$.
If we write $2\tilde{u}=u$, then $\frac1u\in(\frac14,\frac12]$ and we have 
\[
\big\| W_1 \1_{I_N} T_{N,\varepsilon}^z[ \1_{I_N} W_2] \big\|_{\mathcal{C}^2(L^2(\mathbb{T}^{d+1}))}
\leq C
\|W_1\|_{L^u_tL^2_x(\mathbb{T}^{d+1})}
\|W_2\|_{L^u_tL^2_x(\mathbb{T}^{d+1})}, 
\]
provided $\frac1u= \frac12+\frac12({\rm Re}z-\frac d2),{\rm Re}z\in(\frac{d-1}{2},\frac d2]$.
On the other hand, we claim that for ${\rm Re}z=-1$, 
$T^{z}_{N,\varepsilon}:L^2_{x,t}(\mathbb{T}^d  \times I_N)\to L^2_{x,t}(\mathbb{T}^d \times I_N)$ holds with some constant depending only on $d$ and ${\rm Im}z$ exponentially. 
In fact, from Plancherel's theorem, we have for each $t\in\mathbb{T}$, 
\begin{align*}
\big\| T^z_{N,\varepsilon} F(\cdot,t) \big\|_{L^2_x}^2 
&= 
\sum_{m\in S_{d,N}} \bigg| \int_{\varepsilon<|t'|<N^{-1}} t'^{-1+i{\rm Im}z} e^{-2\pi i (t-t') |m|^2} \mathcal{F}_x [ F(\cdot,t-t') ](m)\, \mathrm{d}t' \bigg|^2 \\ 
&= 
\sum_{m\in S_{d,N}} \bigg| \int_{\varepsilon<|t'|<N^{-1}} t'^{-1+i{\rm Im}z} G_m(t-t')\, \mathrm{d}t' \bigg|^2, 
\end{align*}
where $G_m(s) = e^{-2\pi i s |m|^2} \mathcal{F}_x [ F(\cdot,s) ](m)$. 
So, if we further define $H_{N,\varepsilon}^z: G(t) \mapsto \int_{\varepsilon<|t'|<N^{-1}} t'^{-1+i{\rm Im}z} G(t-t')\, \mathrm{d}t'$, then 
\[
\big\| T^z_{N,\varepsilon} F \big\|_{L^2_{x,t}}^2 
=
\sum_{m\in S_{d,N}} \| H^z_{N,\varepsilon} G_m \|_{L^2_t}^2. 
\]
Therefore, once we have the bound $H^z_{N,\varepsilon}: L^2 \to L^2$ with some constant depending only on ${\rm Im}z$ exponentially, then we obtain the desired bound $T^{z}_{N,\varepsilon}:L^2_{x,t}(\mathbb{T}^d  \times I_N)\to L^2_{x,t}(\mathbb{T}^d \times I_N)$. 
Indeed,  the bound $H^z_{N,\varepsilon}: L^2 \to L^2$ holds true since the operator $H^z_{N,\varepsilon}$ is just Hilbert transform up to $i{\rm Im}z$. For further detail, see Vega \cite{Vega}. 
Hence, using $T^{z}_{N,\varepsilon}:L^2_{x,t}(\mathbb{T}^d  \times I_N)\to L^2_{x,t}(\mathbb{T}^d \times I_N)$, we obtain for ${\rm Re}z=-1$, 
\[\big\| W_1 \1_{I_N} T_{N,\varepsilon}^z[ \1_{I_N} W_2] \big\|_{\mathcal{C}^\infty(L^2(\mathbb{T}^{d+1}))}
\leq C({\rm Im}z)
\|W_1\|_{L^\infty_tL^\infty_x(\mathbb{T}^{d+1})}
\|W_2\|_{L^\infty_tL^\infty_x(\mathbb{T}^{d+1})}.\]
Applying Stein's analytic interpolation, \eqref{0927-1} holds as long as 
\[
\frac 2\beta+\frac d\alpha=1,\quad  \frac{1}{d+2}\leq \frac{1}{\alpha}<\frac{1}{d+1}.
\]
\end{proof}

Once we have \eqref{e:ShortTime}, then the same inequality replacing $I_N$ by an arbitrary interval $I$ whose length is $N^{-1}$ holds true: 
\begin{equation}\label{e:anyinterval}
\bigg\| \sum_j \lambda_j |\mathcal{E}_Na_j|^2\bigg\|_{L^p_tL^q_x(\mathbb{T}^d\times I)}
\leq C\|\lambda\|_{\ell^\frac{2q}{q+1}}
\end{equation}
where the constant $C$ is independent of $I$. 
In fact, if we denote the center of the interval $I$ by $c(I)$, then changing variables give 
\[ \bigg\| \sum_j \lambda_j |\mathcal{E}_Na_j|^2\bigg\|_{L^p_tL^q_x(\mathbb{T}^d\times I)} = \bigg\| \sum_j \lambda_j |\mathcal{E}_Nb_j|^2\bigg\|_{L^p_tL^q_x(\mathbb{T}^d\times I_N)} \]
where $b_j(n) = a_j(n)e^{-2\pi i c(I)|n|^2}$. Since $(b_j)_j$ is orthonormal in $\ell^2$ if $(a_j)_j$ is orthonormal,  \eqref{e:ShortTime} reveals the desired inequality. 
From this observation, we may prove Theorem \ref{t:ONS-mix}. 
\begin{proof}[Proof of Theorem \ref{t:ONS-mix}]
We have a covering $\mathbb{T}=\bigcup_{i=1}^NI_i$ where $\{I_i\}_{i=1}^N$ is the collection of disjoint intervals whose length is $N^{-1}$ and decompose 
\[ \bigg\| \sum_j \lambda_j |\mathcal{E}_Na_j|^2\bigg\|_{L^p_tL^q_x(\mathbb{T}^{d+1})}^p = \sum_{i=1}^N \bigg\| \sum_j \lambda_j |\mathcal{E}_Na_j|^2\bigg\|_{L^p_tL^q_x(\mathbb{T}^d\times I_i)}^p. \]
Applying \eqref{e:anyinterval}, we obtain \eqref{e:ONS-mix}. 
\end{proof}

\subsection{Proof of Theorem \ref{t:endpoint-1d}}
One notices that $\frac{2q}{q+1}=2$ holds if $(\frac1q,\frac1p)=(0,\frac12)$ and this is a key point for the proof of Theorem \ref{t:endpoint-1d}. So, the desired inequality \eqref{e:ONS-d1} is equivalent to 
\begin{equation}\label{e:critical-1d}
\bigg\| \sum_j \lambda_j |\mathcal{E}_Na_j|^2\bigg\|_{L^2_tL^\infty_x(\mathbb{T}^{2})} 
\lesssim N^\frac12 \|\lambda\|_{\ell^2}.
\end{equation}
\begin{proof}[Proof of Theorem \ref{t:endpoint-1d}]
From Lemma \ref{l:duality}, \eqref{e:critical-1d} is equivalent to 
\begin{equation}\label{e:goal-1d}
\| W_1 \mathcal{E}_N\mathcal{E}_N^* W_2 \|_{\mathcal{C}^2(L^2(\mathbb{T}^{2}))} \lesssim N^\frac12 \|W_1\|_{L^4_tL^2_x(\mathbb{T}^{2})} \|W_2\|_{L^4_tL^2_x(\mathbb{T}^{2})}.
\end{equation}
%Here, recall that the operator $\mathcal{E}_N\mathcal{E}_N^*$ can be written in terms of the integral kernel $K_N(x,t)$:
%\[ \mathcal{E}_N\mathcal{E}_N^* F(x,t) = \int_{\mathbb{T}^{1+1}} K_N(x-x',t-t') F(x',t')\, \mathrm{d}x'\mathrm{d}t',\quad  K_N(x,t) = \sum_{n=-N}^N e^{2\pi i (xn + t|n|^2)}. \]
Recalling \eqref{e:convolution}, we see that the left-hand side of \eqref{e:goal-1d} turns into 
\begin{align*}
&\big\| W_1 \mathcal{E}_N\mathcal{E}_N^* W_2 \big\|_{\mathcal{C}^2(L^2(\mathbb{T}^{2}))}^2\\
= &
\int_{\mathbb{T}^{2}} \int_{\mathbb{T}^{2}} |W_1(x,t) K_N(x-x',t-t') W_2(x',t')|^2\, \mathrm{d}t\mathrm{d}x\mathrm{d}t'\mathrm{d}x'.
\end{align*}
Now, we expand $|K_N(x-x',t-t')|^2$ as follows.
\[ 
|K_N(x-x',t-t')|^2 
= 
\sum_{n_1,n_2=-N}^N e^{2\pi i [ (x-x')(n_1-n_2) + (t-t')(|n_1|^2 - |n_2|^2) ]}.
\]
If we write $|W_i|^2 = \psi_i$, then 
\begin{align*}
&\| W_1 \mathcal{E}_N\mathcal{E}_N^* W_2 \|_{\mathcal{C}^2(L^2(\mathbb{T}^{2}))}^2
%=& 
%\sum_{n_1,n_2=-N}^N
%\int_{\mathbb{T}^{1+1}} \int_{\mathbb{T}^{1+1}} 
%\psi_1(x,t) e^{2\pi i [ (x-x')(n_1-n_2) + (t-t')(|n_1|^2 - |n_2|^2) ]} \psi_2(x',t')\, \mathrm{d}t\mathrm{d}x\mathrm{d}t'\mathrm{d}x'\\
=
{\rm I} + {\rm II}, 
\end{align*}
where I is the case when $n_1=n_2$:
\[ 
{\rm I} = \sum_{n=-N}^N \int_{\mathbb{T}^{2}} \int_{\mathbb{T}^{2}} 
\psi_1(x,t)  \psi_2(x',t')\, \mathrm{d}t\mathrm{d}x\mathrm{d}t'\mathrm{d}x', 
\]
and II is the case when $n_1\neq n_2$:
\[
{\rm II} = \sum_{n_1\neq n_2}
\int_{\mathbb{T}^{2}} \int_{\mathbb{T}^{2}} 
\psi_1(x,t) e^{2\pi i [ (x-x')(n_1-n_2) + (t-t')(|n_1|^2 - |n_2|^2) ]} \psi_2(x',t')\, \mathrm{d}t\mathrm{d}x\mathrm{d}t'\mathrm{d}x'. 
\]
We first handle II. Rewrite 
\begin{align*} 
&\sum_{n_1\neq n_2} e^{2\pi i [ (x-x')(n_1-n_2) + (t-t')(|n_1|^2 - |n_2|^2) ]} \\ 
=& 
\sum_{\substack{m_1=-2N,\\ m_1\neq  0}}^{2N} \sum_{m_2=-N^2}^{N^2} e^{2\pi i [ (x-x')m_1 + (t-t')m_2 ]}
\sum_{\substack{n_1\neq n_2: \\ n_1-n_2=m_1, |n_1|^2-|n_2|^2=m_2}} 1\\
=& 
\sum_{\substack{m_1=-2N,\\ m_1\neq  0}}^{2N} \sum_{m_2=-N^2}^{N^2} e^{2\pi i [ (x-x')m_1 + (t-t')m_2 ]}
\1_{S_{2,N}}\Big(2^{-1}(m_1+\frac{m_2}{m_1}), 2^{-1}(-m_1+\frac{m_2}{m_1}) \Big),
\end{align*}
since the number of $(n_1,n_2)$ satisfying the condition $n_1\neq n_2$, $n_1-n_2=m_1$ and $|n_1|^2-|n_2|^2=m_2$ for fixed $m_1\neq0,m_2$ is at most one. %$m_1\neq0$ and $|n_1|^2-|n_2|^2=(n_1+n_2)(n_1-n_2)$.
%%%%%%%%%%%%%%%%%%%%%%%%%%%%%%%%%%%%%%%%%%%%%
%footnote
%\footnote{
%In fact, for fixed $m_1\neq0$ and $m_2$, 
%if $n_1,n_2$ satisfy $n_1-n_2=m_1, |n_1|^2-|n_2|^2=m_2$, then $n_1$ and $n_2$ are determined uniquely. Namely, $n_1=\frac12(m_1+\frac{m_2}{m_1})$ and $n_2=\frac12(-m_1+\frac{m_2}{m_1})$. So, this is possible only when $\frac12(m_1+\frac{m_2}{m_1}), \frac12(-m_1+\frac{m_2}{m_1}) \in \mathbb{Z}$. 
%}
%footnote
%%%%%%%%%%%%%%%%%%%%%%%%%%%%%%%%%%%%%%%%%%%%%
For the sake of simplicity, we write $m_2 \in M_{N}(m_1)$ if $m_2 \in [-N^2,N^2]$ and $2^{-1}(m_1+\frac{m_2}{m_1}),2^{-1}(-m_1+\frac{m_2}{m_1})\in \mathbb{Z}\cap [-N,N]$.
From this observation, 
\begin{align*}
{\rm II} 
&= 
\sum_{\substack{m_1=-2N,\\ m_1\neq  0}}^{2N} \sum_{m_2\in M_N(m_1)} 
\int_{\mathbb{T}^{2}} \int_{\mathbb{T}^{2}} 
\psi_1(x,t) e^{2\pi i [ (x-x')m_1 + (t-t')m_2 ]} \psi_2(x',t')\, \mathrm{d}t\mathrm{d}x\mathrm{d}t'\mathrm{d}x'\\
&=
\sum_{\substack{m_1=-2N,\\ m_1\neq  0}}^{2N} \sum_{m_2\in M_N(m_1)} 
\overline{\widehat{\psi_1}(m_1,m_2)} \cdot \widehat{\psi_2}(m_1,m_2) \\
&\leq 
\sum_{\substack{m_1=-2N,\\ m_1\neq  0}}^{2N} \Big( \sum_{m_2\in \mathbb{Z}} 
|\widehat{\psi_1}(m_1,m_2)|^2 \Big)^\frac12 
\Big(\sum_{m_2\in \mathbb{Z}} 
|\widehat{\psi_2}(m_1,m_2)|^2 \Big)^\frac12. 
\end{align*}
If we use the notation $\mathcal{F}_x \psi_1 (m_1,t) = \int_{\mathbb{T}} e^{-2\pi i xm_1} \psi_1(x,t)\, \mathrm{d}x$, then we clearly have $\widehat{\psi_1}(m_1,m_2) = \mathcal{F}_t[\mathcal{F}_x\psi_1 (m_1, \cdot)](m_2) $. 
Applying the Plancherel and the Hausdorff-Young which states that $\mathcal{F}_x: L^1(\mathbb{T}) \to \ell^\infty$, 
\begin{align*}  
\Big( \sum_{m_2\in \mathbb{Z}} 
|\widehat{\psi_1}(m_1,m_2)|^2 \Big)^\frac12 
%&=
%\Big( \sum_{m_2\in \mathbb{Z}} 
%|\mathcal{F}_t[\mathcal{F}_x\psi_1 (m_1, \cdot)](m_2)|^2 \Big)^\frac12 \\
&=
\Big(
\int_{\mathbb{T}} |\mathcal{F}_x\psi_1 (m_1, t)|^2\, \mathrm{d}t
\Big)^\frac12 \\ 
%&\leq 
%\Big(
%\int_{\mathbb{T}} \sup_{m_1}|\mathcal{F}_x\psi_1 (m_1, t)|^2\, \mathrm{d}t
%\Big)^\frac12 \\ 
&\leq 
\bigg(
\int_{\mathbb{T}} \Big(\int_{\mathbb{T}}|\psi_1 (x, t)|\, \mathrm{d}x\Big)^2\, \mathrm{d}t
\bigg)^\frac12. 
\end{align*} 
Putting together with $\psi_i = |W_i|^2$, we see 
\[
{\rm II} \leq %4N \| \psi_1 \|_{L^2_tL^1_x(\mathbb{T}^{2})} \| \psi_2 \|_{L^2_tL^1_x(\mathbb{T}^{2})} 
%= 
4N \|W_1\|_{L^4_tL^2_x(\mathbb{T}^{2})}^2 \|W_1\|_{L^4_tL^2_x(\mathbb{T}^{2})}^2.
\]
 
On the other hand, for $I$, we easily have from H\"{o}lder, 
\[
I %= 2N \|W_1\|_{L^2_{t,x}(\mathbb{T}^{1+1})}^2 \|W_2\|_{L^2_{t,x}(\mathbb{T}^{1+1})}^2 
\leq 2N \|W_1\|_{L^4_tL^2_x(\mathbb{T}^{2})}^2 \|W_1\|_{L^4_tL^2_x(\mathbb{T}^{2})}^2. 
\]
In total, 
\[
\big\| W_1 \mathcal{E}_N\mathcal{E}_N^* W_2 \big\|_{\mathcal{C}^2(L^2(\mathbb{T}^{2}))}^2 
\leq 6N \|W_1\|_{L^4_tL^2_x(\mathbb{T}^{2})}^2 \|W_1\|_{L^4_tL^2_x(\mathbb{T}^{2})}^2,
\]
which implies \eqref{e:goal-1d}.
\end{proof}

\section{The well-posedness of the Hartree equation \eqref{e:Hartree}}

In this section, we prove Theorem \ref{t:wellposed} applying our orthonormal Strichartz inequalities. 
We obtained the orthonormal inequality in the form of \eqref{e:ONS-Td} in the previous sections. By the same proof, it is also possible to replace $P_{\leq N}$ by $P_k$ for any $k\in\mathbb{N}\cup\{0\}$. 
For example, Theorem \ref{t:ONS-pure} can be rephrased by for any $k\in\mathbb{N}\cup\{0\}$, 
\[ \bigg\| \sum_{j}\lambda_j |e^{it\Delta}P_{k} f_j|^2 \bigg\|_{L^{p}_{t}L^q_x(\mathbb{T}^{d+1})} \leq C_\rho 2^{k\rho} \|\lambda\|_{\ell^\alpha}. \]
Keeping this in mind, we give a more general result which can be derived by assuming 
\begin{equation}\label{e:generalONS}
\bigg\| \sum_{j}\lambda_j |e^{it\Delta}P_{k} f_j|^2 \bigg\|_{L^p_tL^q_x(\mathbb{T}^{d+1})} \leq C_\rho 2^{k\rho} \|\lambda\|_{\ell^\alpha},\quad (k\in\mathbb{N}\cup\{0\}).
\end{equation}

\begin{proposition}\label{p:generalwellposed}
Suppose \eqref{e:generalONS} for some $p,q,\alpha\in[1,\infty]$ and some $\rho>0$. Let $s>\frac\rho2$ and $w\in B^s_{q',\infty}$. 
\begin{enumerate}
\item
For any $\gamma_0\in\mathcal{C}^{\alpha,s}(L^2)$ with $R:=\|\gamma_0\|_{\mathcal{C}^{\alpha,s}(L^2)}<\infty$, there exists $T=T(R, \|w\|_{B^s_{q',\infty}} ) > 0$ and $\gamma\in C^0_t([0,T];\mathcal{C}^{\alpha,s}(L^2))$ satisfying \eqref{e:Hartree} on $[0,T]\times\mathbb{T}^d$ and $\rho_\gamma\in L^p_tL^q_x([0,T]\times\mathbb{T}^d)$.
\item
For each $T>0$, we have $R_T=R_T(\| w \|_{B^s_{q',\infty}} )$ such that if $\|\gamma_0\|_{\mathcal{C}^{\alpha,s}(L^2)} \leq R_T$, then there exists a solution $\gamma \in C^0_t([0,T]; \mathcal{C}^{\alpha,s}(L^2))$ satisfying \eqref{e:Hartree} on $[0,T]\times \mathbb{T}^d$ and $\rho_\gamma\in L^p_tL^q_x([0,T]\times\mathbb{T}^d)$. 
\end{enumerate}
\end{proposition}

Once we have Proposition \ref{p:generalwellposed}, then it suffices to combine this with Theorem \ref{t:ONS-mix} to have Theorem \ref{t:wellposed}. 
In fact, using Proposition \ref{p:generalwellposed} with $(\frac1q,\frac1p)\in(A,B)$, $\rho=\frac1p$, $w=w_a$ and $\alpha=\frac{2q}{q+1}$, we obtain Theorem \ref{t:wellposed} since the assumption of Proposition \ref{p:generalwellposed} can be ensured by Theorem \ref{t:ONS-mix} and $w_a\in B^s_{q',\infty}$ holds if $a\leq \frac{d}{q'}-s=\frac2p-s$. 
So, from now on, we prove Proposition \ref{p:generalwellposed} following the argument due to Frank-Sabin \cite[Theorem 14]{frank-sabin-1} with few twists.  
Our ingredient is the part of the control of the nonlinearity where we employ the estimate involving the Besov space $B^s_{q',\infty}$.
%We will give the proof of this proposition here for the sake of the completeness although many parts are overlapped with the proof of Frank-Sabin \cite[Theorem 14]{frank-sabin-1}. Only one difference is the part of the control of the nonlinearity where we employ the estimate involving the Besov space $B^s_{q',\infty}$. 

As a direct corollary of \eqref{e:generalONS}, we have for any $\varepsilon>0$, any $\lambda\in\ell^\alpha$ and any orthonormal system $(f_j)_j$ in $L^2$, 
\begin{equation}\label{e:globalONS}
\bigg\| \sum_{j}\lambda_j |e^{it\Delta}\langle D \rangle^{-\frac\rho2 -\varepsilon} f_j|^2 \bigg\|_{L^p_tL^q_x(\mathbb{T}^{d+1})} \leq C_{\rho,\varepsilon} \|\lambda\|_{\ell^\alpha}. 
\end{equation}
%or equivalently, for any orthonormal system $(f_j)_j$ in $H^{\frac\rho2+\varepsilon}$, 
%\begin{equation}\label{e:globalONS}
%\bigg\| \sum_{j}\lambda_j |e^{it\Delta} f_j|^2 \bigg\|_{L^p_tL^q_x(\mathbb{T}^d)} \leq C_{\rho,\varepsilon} \|\lambda\|_{\ell^\alpha}. 
%\end{equation} 

In fact, using the vector-valued version of the Littlewood-Paley theorem (for example, Lemma 1 in \cite{sabin-2}) and \eqref{e:generalONS}, we obtain 
\begin{align*}
&\bigg\|\sum_{j}\lambda_j|e^{it\Delta}\langle D\rangle^{-(\frac\rho2+\varepsilon)}f_j|^2\bigg\|_{L^p_tL^q_x(\mathbb{T}^{d+1})}\\
\lesssim &
\bigg\|\sum_{j}\lambda_j|e^{it\Delta}P_{0}f_j|^2\bigg\|_{L^p_tL^q_x(\mathbb{T}^{d+1})}
+
\bigg\|\sum_{k=1}^\infty\sum_{j}\lambda_j|2^{-k(\frac\rho2+\varepsilon)}e^{it\Delta}P_kf_j|^2\bigg\|_{L^p_tL^q_x(\mathbb{T}^{d+1})}\\
\lesssim &
\|\lambda\|_{\ell^{\alpha}}
+
\sum_{k=1}^\infty2^{-k(\rho+2\varepsilon)}
\bigg\|\sum_j\lambda_j |e^{it\Delta}P_{k}f_j|^2\bigg\|_{L^p_tL^q_x(\mathbb{T}^{d+1})}
\lesssim_\varepsilon
\|\lambda\|_{\ell^{\alpha}},
\end{align*}
as we desired. 

In the sequel, we denote $s=\frac\rho2+\varepsilon$. 
Before going to the next step, let us recall about the density function, although we do not give the complete treatment of the density function of $\gamma$ here. We refer to \cite{FLLS} for further detail. 
A concrete example of our interest is $\rho_{e^{it\Delta} \langle D \rangle^{-s}\gamma_0 \langle D \rangle^{-s} e^{it\Delta}} (x) = \sum_j \lambda_j|e^{it\Delta} \langle D \rangle^{-s} f_j(x)|^2$ where $\gamma_0=\sum_j\lambda_j |f_j\rangle\langle f_j|$, $(f_j)_j$ is the orthonormal system in $L^2(\mathbb{T}^d)$. Then the density function $\rho_{e^{it\Delta} \langle D \rangle^{-s}\gamma_0 \langle D \rangle^{-s} e^{it\Delta}} (x)$ satisfies 
\begin{equation}\label{e:density-prop}
\int_{\mathbb{T}^d} \rho_{e^{it\Delta} \langle D \rangle^{-s}\gamma_0 \langle D \rangle^{-s} e^{it\Delta}} (x) V(x)\, \mathrm{d}x = {\rm Tr}_{L^2(\mathbb{T}^d)} (\gamma_0 e^{-it\Delta}\langle D \rangle^{-s} V \langle D \rangle^{-s} e^{it\Delta})
\end{equation}
for any nice function $V:\mathbb{T}^d\to [0,\infty)$.
From the definition, it is clear that \eqref{e:globalONS} is equivalent to 
\begin{equation}\label{e:density}
\left\|\rho_{e^{it\Delta} \langle D \rangle^{-s} \gamma_0 \langle D \rangle^{-s} e^{-it\Delta}}\right\|_{L^p_tL^q_x(\mathbb{T}^{d+1})}
\leq C_{\rho,\varepsilon} \|\gamma_0\|_{\mathcal{C}^\alpha(L^2)}, \quad \gamma_0\in\mathcal{C}^\alpha(L^2).
\end{equation}

\begin{proposition}\label{p:inhomONS}
\

\begin{enumerate}
\item
The orthonormal Strichartz inequality \eqref{e:globalONS} or \eqref{e:density} is equivalent to for any $V\in L^{p'}_tL^{q'}_x(\mathbb{T}^{d+1})$, 
\begin{equation}\label{e:dual-ons}
\bigg\| \int_{\mathbb{T}}e^{-it\Delta}\langle D \rangle^{-s}V(x,t) \langle D \rangle^{-s}e^{it\Delta}\, \mathrm{d}t \bigg\|_{\mathcal{C}^{\alpha'}(L^2)}
\leq C_{\rho,\varepsilon}
\|V\|_{L^{p'}_tL^{q'}_x(\mathbb{T}^{d+1})}.
\end{equation}
\item(Inhomogeneous estimate)
Let $R(t'):L^2\to L^2$ be self-adjoint for each $t'\in\mathbb{T}$ and define 
\[
\gamma(t)=\int_0^te^{i(t-t')\Delta}R(t')e^{i(t'-t)\Delta}\, \mathrm{d}t',\quad (t\in\mathbb{T}).
\]
Suppose one of \eqref{e:globalONS}, \eqref{e:density} and \eqref{e:dual-ons} holds true. Then 
\begin{equation}\label{e:inhom}
\|\rho_{\langle D \rangle^{-s} \gamma(t) \langle D \rangle^{-s}} \|_{L^p_tL^q_x(\mathbb{T}^{d+1})}
\leq C_{\rho,\varepsilon}
\bigg\|\int_{\mathbb{T}}e^{-is\Delta}|R(s)|e^{is\Delta}\, \mathrm{d}s\bigg\|_{\mathcal{C}^\alpha(L^2)}.
\end{equation}

\end{enumerate}
\end{proposition}

\begin{proof}
Since the proof of this proposition is almost the same as in \cite{FLLS,frank-sabin-1}, we omit details and give key steps. 
To show \eqref{e:dual-ons}, in view of the duality, we have only to show 
\begin{equation}\label{e:1214-1}
\bigg|{\rm Tr}_{L^2}\bigg( \gamma_0 \int_{\mathbb{T}}e^{-it\Delta}\langle D \rangle^{-s}V(x,t) \langle D \rangle^{-s}e^{it\Delta}\, \mathrm{d}t \bigg)\bigg| \lesssim \|V\|_{L^{p'}_tL^{q'}_x(\mathbb{T}^{d+1})}
\end{equation}
for any $\gamma_0: \|\gamma_0 \|_{\mathcal{C}^\alpha(L^2)}=1$ which follows from the combination of $\eqref{e:density-prop}$ and \eqref{e:density}.

To show \eqref{e:inhom}, we notice from the duality and the property of the density function that for some non-negative function $V=V(x,t)$ such that $\|V\|_{L^{p'}_tL^{q'}_x(\mathbb{T}^{d+1})} = 1$, 
\begin{align*}
&\|\rho_{\langle D \rangle^{-s} \gamma(t) \langle D \rangle^{-s}} \|_{L^p_tL^q_x(\mathbb{T}^{d+1})}
= 
\int_{\mathbb{T}} {\rm Tr}_{L^2} ( \gamma(t) \langle D\rangle^{-s} V(t) \langle D\rangle^{-s} )\, \mathrm{d}t\\ 
\leq& 
\bigg\| \int_{\mathbb{T}} e^{-it\Delta} \langle D\rangle^{-s} V(t) \langle D\rangle^{-s} e^{it\Delta}\, \mathrm{d}t \bigg\|_{\mathcal{C}^{\alpha'}(L^2)} 
\bigg\| \int_{\mathbb{T}} e^{-it'\Delta} |R(t')| e^{it'\Delta}\, \mathrm{d}t' \bigg\|_{\mathcal{C}^{\alpha}(L^2)}, 
\end{align*}
where we used the fact that $|{\rm Tr}_{L^2}(AB)| \leq {\rm Tr}_{L^2} (|A| |B|)$ for self-adjoint opeartors $A,B$. So, applying \eqref{e:dual-ons}, we obtain \eqref{e:inhom}.
\end{proof}

Note that from Duhamel's principle the solution of the inhomogeneous equation  
\begin{equation}\label{e:inhomHartree}
  \left\{
  \begin{array}{ll}
	i\partial_t\gamma=[-\Delta,\gamma] + R(t),  \quad (x,t)\in \mathbb{T}^d \times \mathbb{R}  \\
	\gamma|_{t=0}=\gamma_0,
  \end{array} %\right.
  \right.
\end{equation}
can be written by  
\[  
e^{it\Delta}\gamma_0e^{-it\Delta}
-i\int_0^t e^{i(t-t')\Delta}R(t')e^{i(t'-t)\Delta}\, \mathrm{d}t'.
\]
So, the inequality \eqref{e:inhom} is an estimate of the inhomogeneous term.

Remark that \eqref{e:density} and \eqref{e:inhom} can be generalize: for any $T>0$, 
\begin{equation}\label{e:density-T}
\left\|\rho_{e^{it\Delta} \langle D \rangle^{-s} \gamma_0 \langle D \rangle^{-s} e^{-it\Delta}}\right\|_{L^p_tL^q_x([0,T]\times\mathbb{T}^{d})}
\leq C_{\rho,\varepsilon} T^{1/p} \|\gamma_0\|_{\mathcal{C}^\alpha(L^2)}, 
\end{equation}
and 
\begin{equation}\label{e:inhom-T}
\|\rho_{\langle D \rangle^{-s} \gamma(t) \langle D \rangle^{-s}} \|_{L^p_tL^q_x([0,T]\times\mathbb{T}^{d})}
\leq C_{\rho,\varepsilon} T^{1/p} 
\bigg\|\int_{\mathbb{T}}e^{-is\Delta}|R(s)|e^{is\Delta}\, \mathrm{d}s\bigg\|_{\mathcal{C}^\alpha(L^2)}.
\end{equation}

Now, we prove Proposition \ref{p:generalwellposed} using Proposition \ref{p:inhomONS}. 
\begin{proof}[Proof of Proposition \ref{p:generalwellposed}]
First we prove the local well-posedness Proposition \ref{p:generalwellposed}-(1). 
Let us write $\|\gamma_0\|_{\mathcal{C}^{\alpha,s}(L^2)} = R <\infty$ and take $T=T(R, \|w\|_{B^s_{q',\infty}})\leq1$ to be chosen later. 
To capture the solution by employing the fixed point theorem, define the space $X$ by 
\begin{align*}
X_T=
\{
(\gamma,\rho)\in C^0_t([0,T];\mathcal{C}^{\alpha,s}(L^2))\times L^p_tL^q_x([0,T]\times\mathbb{T}^d): \| (\gamma, \rho) \|_{X_T} \leq C^*R\},
\end{align*}
where 
\[
\| (\gamma,\rho) \|_{X_T} := \|\gamma\|_{C^0_t([0,T];\mathcal{C}^{\alpha,s}(L^2))}+\|\rho\|_{L^p_tL^q_x([0,T]\times\mathbb{T}^d)}
\]
and $C^*$ is chosen so that $C^*>\max{(10,10C_{\rho,\varepsilon})}$. 
Next, define the contraction map $\Phi$. First, define 
\[
\Phi_1(\gamma,\rho)(t)
=
e^{it\Delta}\gamma_0e^{-it\Delta}
-i\int_0^t e^{i(t-t')\Delta}[w_a\ast \rho(t'),\gamma(t')]e^{i(t'-t)\Delta}\, \mathrm{d}t'
\] 
and 
\[
\Phi(\gamma,\rho)
=
(\Phi_1(\gamma,\rho),\rho[\Phi_1(\gamma,\rho)]). 
\]
Here, we used the notation $\rho[\gamma]=\rho_\gamma$. In this formulation, \eqref{e:Hartree} is equivalent to $(\gamma,\rho_\gamma) = \Phi(\gamma,\rho_\gamma)$. 
We now claim that for any $T>0$ and any small $\delta>0$, 
\begin{equation}\label{e:contract-1}
\|\Phi_1(\gamma,\rho)\|_{C^0_t([0,T];\mathcal{C}^{\alpha,s}(L^2))}
\leq 
R
+
C_{s,\delta}T^{1/p'}\|w\|_{{B}^{s+\delta}_{q',\infty}}(C^*R)^2
\end{equation}
and recalling $C_{\rho,\varepsilon}$ is the constant of the orthonormal Strichartz inequality \eqref{e:globalONS}, 
\begin{equation}\label{e:contract-2}
\|\rho[\Phi_1(\gamma,\rho)]\|_{L^p_tL^q_x([0,T]\times\mathbb{T}^d)}
\leq
C_{\rho,\varepsilon}T^{1/p}\big\{ R+C_{s,\delta}T^{1/p'}\|w\|_{{B}^{s+\delta}_{q',\infty}}(C^*R)^2\big\}.
\end{equation}
Once these claims are proved, then choosing $T\leq1$ small enough so that 
\[
C_{s,\delta}C_{\rho,\varepsilon}T^{1/p'}\|w\|_{{B}^{s+\delta}_{q',\infty}}(C^*R)^2\leq \frac{C^*R}{4},
\]
we see that $\Phi(\gamma,\rho)\in X_T$ for $(\gamma,\rho)\in X_T$ (precisely speaking, $T$ depends on $\|w\|_{{B}^{s+\delta}_{q',\infty}}$, not $\|w\|_{{B}^{s}_{q',\infty}}$, but this is harmless since $s = \rho + \varepsilon$ and $\varepsilon,\delta$ are arbitrary small).
Similarly, we can show that $\Phi$ is a contraction mapping. So, we find a solution to the Hartree equation \eqref{e:Hartree} on $[0,T]$. 

Let us prove \eqref{e:contract-1}. 
To evaluate $\|\Phi_1(\gamma,\rho)\|_{C^0_t([0,T];\mathcal{C}^{\alpha,s}(L^2))}$, fix any $t\in[0,T]$ and calculate 
\begin{align*}
&\|\Phi_1(\gamma,\rho)(t)\|_{\mathcal{C}^{\alpha,s}(L^2)}\\
\leq&
\|e^{it\Delta}\gamma_0e^{-it\Delta}\|_{\mathcal{C}^{\alpha,s}(L^2)}
+
\int_0^T
\big\|e^{i(t-t')\Delta}[w\ast \rho(t'),\gamma(t')]e^{i(t'-t)\Delta}\big\|_{\mathcal{C}^{\alpha,s}(L^2)}\, \mathrm{d}t'.
\end{align*}
The first term is easy to handle since if $(f_j)_j$ is orthonormal in $L^2$, then $(e^{it\Delta}f_j)_j$ is as well for each $t$:
%%%%%%%%%%%%%%%%%%%%%
%footnote
%\footnote{
%In fact, if we write $\gamma_0=\sum_\nu\lambda_\nu |f_\nu\rangle\langle f_\nu|$ where $(f_\nu)_\nu$ is an orthonormal system in $\dot{H}^s(\mathbb{R}^d)$, 
%then $\|\gamma_0\|_{\mathcal{C}^\alpha(\dot{H}^s)}$ and 
%\[
%e^{it\Delta}\gamma_0e^{-it\Delta}=\sum_\nu\lambda_\nu|e^{it\Delta}f_\nu\rangle\langle e^{it\Delta} f_\nu|. 
%\]
%Since $(e^{it\Delta}f_\nu)_\nu$ is an orthonormal system in $\dot{H}^s(\mathbb{R}^d)$ as well for each fixed $t$, the above decomposition gives the orthonormal decomposition of the operator $e^{it\Delta}\gamma_0e^{-it\Delta}$. In particular, the eigenvalues of $e^{it\Delta}\gamma_0e^{-it\Delta}$ are also $(\lambda_\nu)_\nu$. 
%This implies 
%\[
%\|e^{it\Delta}\gamma_0e^{-it\Delta}\|_{\mathcal{C}^\alpha(\dot{H}^s)}=
%\|\lambda\|_{\ell^\alpha}=
%\|\gamma_0\|_{\mathcal{C}^\alpha(\dot{H}^s)}.
%\]
%}
%footnote
%%%%%%%%%%%%%%%%%%%%%
\[
\|e^{it\Delta}\gamma_0e^{-it\Delta}\|_{\mathcal{C}^{\alpha,s}(L^2)}=
\|\gamma_0\|_{\mathcal{C}^{\alpha,s}(L^2)}= R.
\]
For the second term, 
we use the H\"{o}lder inequality for Schatten spaces 
%%%%%%%%%%%%%%%%%%%%%
%footnote
%\footnote{
%In general, we have
%\[
%\|\gamma_1\cdot\gamma_2\|_{\mathcal{C}^{\alpha}(\mathcal{H})}
%\leq
%\|\gamma_1\\|_{\mathcal{C}^{\alpha_1}(\mathcal{H})}
%\|\gamma_2\\|_{\mathcal{C}^{\alpha_2}(\mathcal{H})}
%\]
%provided $1/\alpha=1/\alpha_1+1/\alpha_2$.
%In fact, ....
%}
%footnote
%%%%%%%%%%%%%%%%%%%%%
to have 
\begin{align*}
&\big\|e^{i(t-t')\Delta}[w\ast \rho(t'),\gamma(t')]e^{i(t'-t)\Delta}\big\|_{\mathcal{C}^{\alpha,s}(L^2)}\\
%\leq&
%\| [\langle D \rangle^s w\ast \rho(t') \langle D \rangle^{-s} ] \langle D \rangle^s \gamma(t') \langle D \rangle^s \|_{\mathcal{C}^{\alpha}(L^2)}\\
%&+ \| \langle D \rangle^s \gamma(t') \langle D \rangle^s [\langle D \rangle^{-s} w\ast \rho(t') \langle D \rangle^s ]\|_{\mathcal{C}^{\alpha}(L^2)}\\
\leq&
\big\{ \| \langle D \rangle^s w\ast \rho(t') \langle D \rangle^{-s} \|_{\mathcal{C}^\infty(L^2)} 
+
\| \langle D \rangle^{-s} w\ast \rho(t') \langle D \rangle^{s} \|_{\mathcal{C}^\infty(L^2)} \big\} \| \gamma(t') \|_{\mathcal{C}^{\alpha,s}(L^2)} 
%\left\|e^{i(t-s)\Delta}(w\ast \rho(s)\gamma(s))e^{i(s-t)\Delta}\right\|_{\mathcal{C}^\frac{2q}{q+1}({H}^{\frac1p+\varepsilon})}
%+
%\left\|e^{i(t-s)\Delta}(\gamma(s)w\ast \rho(s))e^{i(s-t)\Delta}\right\|_{\mathcal{C}^\frac{2q}{q+1}({H}^{\frac1p+\varepsilon})}\\
%=&
%\left\|w\ast \rho(s)\gamma(s)\right\|_{\mathcal{C}^\frac{2q}{q+1}({H}^{\frac1p+\varepsilon})}
%+
%\left\|\gamma(s)w\ast \rho(s)\right\|_{\mathcal{C}^\frac{2q}{q+1}({H}^{\frac1p+\varepsilon})}\\
%\leq
%\|w\ast \rho(t')\|_{\mathcal{C}^\infty({H}^{\frac\rho2+\varepsilon})}\|\gamma(s)\|_{\mathcal{C}^\alpha({H}^{\frac\rho2+\varepsilon})}.
\end{align*}
The estimate we employ to evaluate the above nonlinear term is the following (see Corollary on p. 205 in \cite{Triebel} where the inequality was proved for $\mathbb{R}^d$ case, but the same proof is applicable for $\mathbb{T}^d$ case)
\begin{equation}\label{eq:1012-8}
\|f\cdot g\|_{{H}^{r}}
\leq
C_{s,\delta}
\|f\|_{{B}^{|r|+\delta}_{\infty,\infty}}
\|g\|_{{H}^{r}},
\end{equation}
%\[
%\|f\cdot g\|_{{F}^s_{p,q}(\mathbb{T}^d)}
%\leq
%C_{p,q,d,s,\varepsilon}
%\|f\|_{{B}^{s+\varepsilon}_{\infty,\infty}(\mathbb{T}^d)}
%\|g\|_{{F}^s_{p,q}(\mathbb{T}^d)}
%\]
where $r\in\mathbb{R}$ and $\delta>0$ are arbitrary.
%Or it may be reasonable choice to use Them 5.24 in book-sawano which states 
%\[
%\|f\cdot g\|_{\dot{F}^s_{p_0,q}(\mathbb{R}^d)}
%\leq C_{p,q,d,s,\varepsilon}
%\|f\|_{\dot{F}^s_{p_1,q}(\mathbb{R}^d)}
%\|g\|_{\dot{F}^s_{p_2,q}(\mathbb{R}^d)}
%\]
%for $0<p_0,p_1,p_2<\infty$, $s>0$, $0<q\leq\infty$ such that 
%\[0<\frac{1}{p_i}-\frac{s}{n}<1\quad (i=1,2,3),\quad \frac{1}{p_0}=\frac{1}{p_1}+\frac{1}{p_2}-\frac{s}{d}.\]
%In particular, for $0<s<\frac{d}{2}$,
%\begin{equation}\label{eq:1012-9}
%\|f\cdot g\|_{\dot{H}^s(\mathbb{R}^d)}
%\leq C_{d,s,\varepsilon}
%\|f\|_{\dot{W}^{\frac{n}{s},s}(\mathbb{R}^d)}
%\|g\|_{\dot{H}^s(\mathbb{R}^d)}
%\end{equation}
From this estimate and Young's inequality, 
\begin{align*}
\| \langle D \rangle^s w\ast \rho(t') \langle D \rangle^{-s} \|_{\mathcal{C}^\infty(L^2)}  
\leq 
C_{s,\delta}
%\sup_{g\in {H}^{\frac1p+\varepsilon}}
\|w\ast \rho(t')\|_{{B}^{s+\delta}_{\infty,\infty}}
%\|g\|_{{H}^{\frac1p+\varepsilon}}\\
%&=
%C_{d,p,\varepsilon}
%\sup_{k\in\mathbb{N}_0}2^{k(\frac1p+\tilde{\varepsilon})}\|\mathcal{F}^{-1}\varphi_k\ast w\ast \rho(s)\|_{L^\infty_x(\mathbb{T}^d)}\\
%&\leq
%C_{d,p,\varepsilon}
%\sup_{k\in\mathbb{N}_0}2^{k(\frac1p+\tilde{\varepsilon})}\|\mathcal{F}^{-1}\varphi_k\ast w\|_{L^{q'}_x(\mathbb{T}^d)} \|\rho(s)\|_{L^q_x(\mathbb{T}^d)}\\
%&=
\leq
C_{s,\delta}
\|w\|_{{B}^{s+\delta}_{q',\infty}}\|\rho(t')\|_{L^q_x}.
\end{align*}
Similarly, 
\begin{align*}
\| \langle D \rangle^{-s} w\ast \rho(t') \langle D \rangle^{s} \|_{\mathcal{C}^\infty(L^2)}  
\leq
C_{-s,\delta}
\|w\|_{{B}^{s+\delta}_{q',\infty}}\|\rho(t')\|_{L^q_x}.
\end{align*}
In total, from $(\gamma,\rho)\in X_T$, we estimate the second term by 
\begin{align*}
\int_0^T
\left\|e^{i(t-t')\Delta}[w\ast \rho(t'),\gamma(t')]e^{i(t'-t)\Delta}\right\|_{\mathcal{C}^{\alpha,s}(L^2)}\, \mathrm{d}t'
%\leq&
%2 C_{d,p,\varepsilon}\|w\|_{{B}^{\frac1p+\tilde{\varepsilon}}_{q',\infty}(\mathbb{T}^d)}
%\int_0^T
%\|\rho(s)\|_{L^q_x(\mathbb{T}^d)}
%\|\gamma(s)\|_{\mathcal{C}^\frac{2q}{q+1}({H}^{\frac1p+\varepsilon})}ds\\
%\leq&
%2 C_{d,p,\varepsilon}\|w\|_{{B}^{\frac1p+\tilde{\varepsilon}}_{q',\infty}(\mathbb{T}^d)}T^{1/p'}\|\gamma\|_{C^0_t([0,T];\mathcal{C}^\frac{2q}{q+1}({H}^{\frac1p+\varepsilon}))}\|\rho\|%_{L^p_tL^q_x([0,T]\times\mathbb{T}^d)}
\leq
C_{s,\delta}'\|w\|_{{B}^{s+\delta}_{q',\infty}}T^{1/p'}(C^{*}R)^2.
\end{align*}
where $C_{s,\delta}'=C_{s,\delta} + C_{-s,\delta}$ which shows \eqref{e:contract-1}. %(precisely speaking, we have $\tilde{\varepsilon}$ not $\varepsilon$, but this is harmless since both of $\varepsilon$ and $\tilde{\varepsilon}$ are arbitrary small). 

To show \eqref{e:contract-2}, we employ homogeneous and inhomogeneous orthonormal Strichartz estimates \eqref{e:density-T} and \eqref{e:inhom-T} to have 
%If we let 
%\[ \tilde{\gamma}(t) = -i\int_0^t e^{i(t-t')\Delta}[w_a\ast \rho(t'),\gamma(t')]e^{i(t'-t)\Delta}\, \mathrm{d}t' \]
\begin{align*}
&T^{-1/p}C_{\rho,\varepsilon}^{-1}  \|\rho[\Phi_1(\gamma,\rho)]\|_{L^p_tL^q_x([0,T]\times\mathbb{T}^d)}\\
%\leq
%\|\rho_{e^{it\Delta}\gamma_0e^{-it\Delta}}\|_{L^p_tL^q_x([0,T]\times\mathbb{T}^d)}
%+
%\|\rho_{\tilde{\gamma}(t)}\|_{L^p_tL^q_x([0,T]\times\mathbb{T}^d)}\\
\leq&
\| \langle D \rangle^s \gamma_0 \langle D \rangle^s \|_{\mathcal{C}^\alpha(L^2)}
+
\bigg\|\int_{0}^Te^{-it'\Delta} \langle D \rangle^s |[w_a\ast \rho(t'),\gamma(t')]| \langle D \rangle^s e^{it'\Delta}\, \mathrm{d}t'\bigg\|_{\mathcal{C}^\alpha(L^2)}.
\end{align*}
For the first term, $\| \langle D \rangle^s \gamma_0 \langle D \rangle^s \|_{\mathcal{C}^\alpha(L^2)} = R$. For the second term, we may employ the same argument as \eqref{e:contract-1} and we see \eqref{e:contract-2}.

Let us show proposition \ref{p:generalwellposed}-(2). In this case, we first fix an arbitrary $T>0$. The key estimates are \eqref{e:contract-1} and \eqref{e:contract-2} which have been already proved. These two estimates yield that 
\[
\| \Phi (\gamma,\rho) \|_{X_T} \leq (1+C_{\rho,\varepsilon}T^{1/p}) \big( \| \gamma_0 \|_{\mathcal{C}^{\alpha,s}(L^2)} + C_{s,\delta} T^{1/p'}\|w\|_{B^{s+\delta}_{q',\infty}} \| (\gamma, \rho) \|_{X_T}^2  \big). 
\] 
With this in mind, we choose $R_T = R_T(\| w \|_{B^s_{q',\infty}})$ small enough (precisely speaking, $R_T$ depends on $\| w \|_{B^{s+\delta}_{q',\infty}}$, not $\| w \|_{B^s_{q',\infty}}$, but again this is harmless) so that we can find $M>0$ such that for any $y\in [0, M]$, it holds 
\[ 
(1+C_{\rho,\varepsilon}T^{1/p}) \big( \| \gamma_0 \|_{\mathcal{C}^{\alpha,s}(L^2)} + C_{s,\delta} T^{1/p'} \|w\|_{B^{s+\delta}_{q',\infty}} y^2  \big) \leq M  
\]
as long as $ \| \gamma_0 \|_{\mathcal{C}^{\alpha,s}(L^2)} \leq R_T$. So, if we define the space $X_{T,M}$ by 
\[ 
X_{T,M} := \{ (\gamma,\rho)\in X_T: \| (\gamma, \rho) \|_{X_T} \leq M \},
\]
then we see that $\Phi : X_{T,M} \to X_{T,M}$. By choosing $R_T$ smaller further, we can also show that $\Phi$ is a contraction map on $X_{T,M}$ by the similar way and hence from the fixed point theorem we find a solution $\gamma \in C^0_t([0,T]; \mathcal{C}^{\alpha,s}(L^2))$ satisfying $\rho_\gamma \in L^p_tL^q_x([0,T]\times\mathbb{T}^d)$.
\end{proof}

\section{On the beyond region $[A,C]$}
In this final Section, we give one observation on the beyond region $[A,C]$ when $d\geq2$ and this at least gives almost sharp inequality with $\varepsilon$-loss at the point $A$.

\begin{theorem}\label{t:ONS-beyond}
Let $d\geq2$, $N>1$ and $(a_j)_j$ be any orthonormal system in $\ell^2$.
\begin{enumerate}
\item\label{item:1}
On $(\frac1q,\frac1p)=A$, 
\begin{equation}\label{e:ONS-A}
\bigg\| \sum_j \lambda_j |\mathcal{E}_Na_j|^2  \bigg\|_{L^p_tL^q_x(\mathbb{T}^{d+1})} \leq C_\varepsilon N^{\frac1p + \varepsilon} \|\lambda\|_{\ell^{\alpha(1/p)}}  
\end{equation}
holds true for any $\lambda\in\ell^{\alpha(1/p)}$ and arbitrary small $\varepsilon>0$. 
Moreover, this is sharp up to $\varepsilon$. 
\item\label{item:2}
On $(\frac1q,\frac1p)=C$, 
\begin{equation}\label{e:ONS-C}
\bigg\| \sum_j \lambda_j |\mathcal{E}_Na_j|^2  \bigg\|_{L^p_tL^q_x(\mathbb{T}^{d+1})} \leq C_\varepsilon N^{\frac1p + \frac1d + \varepsilon} \|\lambda\|_{\ell^{\alpha(1/p)}}  
\end{equation}
holds true for any $\lambda\in\ell^{\alpha(1/p)}$ and arbitrary small $\varepsilon>0$. 
\end{enumerate}
\end{theorem}

\begin{remark}
We will show Theorem \ref{t:ONS-beyond} in a more general form: for any $(\frac1q,\frac1p)\in[A,C]$, 
\begin{equation}\label{e:beyond}
\bigg\| \sum_j \lambda_j |\mathcal{E}_N a_j|^2  \bigg\|_{L^p_tL^q_x(\mathbb{T}^{d+1})} \leq C_\varepsilon N^{\frac12(d-1-\frac{d+1}{q}) + \frac1p + \varepsilon} \|\lambda\|_{\ell^\frac{2q}{q+1}}. 
\end{equation}
Note that while \eqref{e:ONS-A} gives an almost sharp estimate at $A$ up to $\varepsilon$, \eqref{e:ONS-C} seems not sharp because of the factor $N^\frac{1}{d}$. 
\end{remark}

\begin{proof}
If we recall the argument which we used to prove Theorem \ref{t:ONS-mix}, then it suffices to show 
\[
\bigg\| \sum_j \lambda_j |\mathcal{E}_N a_j|^2  \bigg\|_{L^p_tL^q_x(\mathbb{T}^{d}\times I_N)} \leq C_\varepsilon N^{\frac12(d-1-\frac{d+1}{q})  + \varepsilon} \|\lambda\|_{\ell^\frac{2q}{q+1}}.
\]
Moreover, in view of Lemma \ref{l:duality}, this inequality follows from  
\begin{equation}\label{e:dual-beyond}
\big\| W_1^N \mathcal{E}_N \mathcal{E}_N^* W_2^N \big\|_{\mathcal{C}^\alpha(L^2(\mathbb{T}^{d+1}))} \lesssim N^{\frac{d+1-\alpha}{\alpha}+\varepsilon} \|W_1\|_{L^\beta_t L^\alpha_x(\mathbb{T}^{d+1})} \|W_1\|_{L^\beta_t L^\alpha_x(\mathbb{T}^{d+1})}
\end{equation}
for $\frac 2\beta +\frac d\alpha =1$ and $d\leq \alpha \leq d+1$ where $W_i^N:= \1_{I_N}(t) W_i $. 
To this end, we decompose the operator $\mathcal{E}_N\mathcal{E}_N^* = K_N \ast$ as follows: for $(x,t)\in\mathbb{T}^d\times I_N$, 
\begin{align*}
\mathcal{E}_N\mathcal{E}_N^* W^N(x,t) %&:=  \int_{\mathbb{T}^d\times I_N} K_N(x-x',t-t') W(x',t')\, \mathrm{d}x'\mathrm{d}t' \\
&= \sum_{j=-\infty}^{{\rm log}_2 (N^{-1})}  \int_{\mathbb{T}^d}\int_{2^{j-1} \leq |t-t'| < 2^{j}}  K_N(x-x',t-t') W(x',t')\, \mathrm{d}x'\mathrm{d}t' \\
%&=: \sum_{j=-\infty}^{{\rm log}_2 (N^{-1})} K_{N,j}\ast W(x,t) 
&= \sum_{j=-\infty}^{{\rm log}_2 (N^{-1})} T_{N,j} W(x,t), 
\end{align*}
where $T_{N,j}=K_{N,j}\ast$ and $K_{N,j}=K_N\1_{2^{j-1}\leq |t| <2^j}$.
Hereafter we evaluate each term $\big\| W_1^N T_{N,j} W_2^N \big\|_{\mathcal{C}^\alpha(L^2(\mathbb{T}^{d+1}))}$ .  
We claim that for any $\sigma\in [2,\infty]$ and any parameters $\mu\in[0,1]$, $\rho\geq4$, 
\begin{align} \label{e:local-dual}
&\big\| W_1^N T_{N,j} W_2^N \big\|_{\mathcal{C}^\alpha(L^2(\mathbb{T}^{d+1}))} \nonumber \\
\lesssim & 2^{j[(\frac12-\frac d2 (1-\mu))\frac 2\alpha+1-\frac 2\alpha]} N^{(d\mu - 2(\frac14-\frac1\rho))\frac 2\alpha} \| W_1 \|_{L^{\frac{\rho\alpha}{2}}_tL^\alpha_x(\mathbb{T}^{d+1})} \| W_2 \|_{L^{\frac{\rho\alpha}{2}}_tL^\alpha_x(\mathbb{T}^{d+1})}
\end{align}
To see this, we consider two cases $\alpha=2$ and $\alpha=\infty$. 

When $\alpha=2$, we employ the kernel estimate: for $(x,t) \in \mathbb{T}^d\times I_{N}$, 
\[
|K_{N,j}(x,t)| \lesssim \min{(|t|^{-\frac{d}{2}}, N^d)} \leq |t|^{-\frac d2 (1-\mu)} N^{d\mu} \, \quad (\mu\in [0,1]).
\]
From this estimate, Young's inequality and H\"{o}lder's inequality, 
\begin{align*}
&\big\| W_1^N T_{N,j} W_2^N \big\|_{\mathcal{C}^2(L^2(\mathbb{T}^{d+1}))}^2 \\
%=& 
%\int_{|t-t'|\sim 2^j} | W^N_1(x,t) K_{N,j}(x-x',t-t') W^N_2(x',t') |^2\, \mathrm{d}t\mathrm{d}x\mathrm{d}t'\mathrm{d}x' \\ 
\lesssim& 
N^{2d\mu} 
\int_{|t-t'|\sim 2^j} \| W^N_1(\cdot,t) \|_{L^2_x(\mathbb{T}^d)}^2 |t-t'|^{-d(1-\mu)} \|W^N_2(\cdot,t') \|_{L^2_x(\mathbb{T}^d)}^2\, \mathrm{d}t\mathrm{d}t'\\ 
%\lesssim& 
%N^{2d\mu} 
%2^{-jd(1-\mu)} \|W_1\|_{L^4_tL^2_x(\mathbb{T}^{d+1})}^2 \| 1_{|\cdot|\sim2^j} \|_{L^1} \|W_2\|_{L^4_tL^2_x(\mathbb{T}^{d+1})}^2 \\
\lesssim&
N^{2d\mu} 
2^{j(1-d(1-\mu))} N^{ -4(\frac14-\frac1\rho) } \|W_1\|_{L^\rho_tL^2_x(\mathbb{T}^{d+1})}^2  \|W_2\|_{L^\rho_tL^2_x(\mathbb{T}^{d+1})}^2 
\end{align*}
holds for any $\rho\geq4$. 

On the other hand, when $\alpha=\infty$, we see from Plancherel's theorem that for any $F\in L^2(\mathbb{T}^{d+1})$
\begin{align*}
\big\| W_1^N T_{N,j} [W_2^N F] \big\|_{L^2(\mathbb{T}^{d+1}))}
%&\leq  
%\|W_1\|_{L^\infty_tL^\infty_x(\mathbb{T}^{d+1})} \| T_{N,j} [W_2^N F] \|_{L^2(\mathbb{T}^{d+1}))} \\
%&=
%\|W_1\|_{L^\infty_tL^\infty_x(\mathbb{T}^{d+1})} \| \mathcal{F}_{x,t} K_{N,j} \cdot  \mathcal{F}_{x,t} [W_2^N F] \|_{\ell^2(\mathbb{Z}^{d+1}))} \\
%&\lesssim 
%2^j
%\|W_1\|_{L^\infty_tL^\infty_x(\mathbb{T}^{d+1})} \|  \mathcal{F}_{x,t} [W_2^N F] \|_{\ell^2(\mathbb{Z}^{d+1}))} \\ 
&\lesssim  
2^j
\|W_1\|_{L^\infty_tL^\infty_x(\mathbb{T}^{d+1})} \|   W_2 \|_{L^\infty_tL^\infty_x(\mathbb{T}^{d+1})}  \| F \|_{L^2(\mathbb{T}^{d+1})},  
\end{align*}
since we have for any $(n,n_{d+1}) \in \mathbb{Z}^{d+1}$, 
%%%%%%%%%%%%%%%%%%%%%%%%%%%%%%%%%
%footnote
%\footnote{
%In fact, 
%\begin{align*}
%\mathcal{F}_{x,t} K_{N,j}(n,n_{d+1})
%&=
%\int_{\mathbb{T}^{d+1}} e^{-2\pi i (x\cdot n + t n_{d+1})} \1_{|t|\sim 2^j} \sum_{m\in S_{d,N}} e^{2\pi i(x\cdot m +t|m|^2)} \, \mathrm{d}x\mathrm{d}t \\
%&= 
%\int_{|t|\sim 2^j } \sum_{m\in S_{d,N}} e^{2\pi i t(|m|^2-n_{d+1})} \delta_{n,m}\, \mathrm{d}t \\
%&=
%\int_{|t|\sim 2^j } e^{2\pi i t(|n|^2-n_{d+1})} \, \mathrm{d}t \lesssim 2^j.
%\end{align*}
%}
%footnote
%%%%%%%%%%%%%%%%%%%%%%%%%%%%%%%%%
\[ |\mathcal{F}_{x,t} K_{N,j}(n,n_{d+1})|  \lesssim 2^j.  \]

Interpolating these two estimates, we obtain \eqref{e:local-dual}. 
To sum up each estimate \eqref{e:local-dual}, we need to 
\[ (\frac12-\frac d2 (1-\mu))\frac 2\alpha+1-\frac 2\alpha >0 \] 
or equivalently, 
$ \mu>\frac{d+1-\alpha}{d} $ which gives the restriction of $\mu$. 
Under this restriction, we can sum up \eqref{e:local-dual} and obtain 
\begin{align*}
&\big\| W_1^N \mathcal{E}_N \mathcal{E}_N^* W_2^N \big\|_{\mathcal{C}^\alpha(L^2(\mathbb{T}^{d+1}))} \\
%&\leq 
%\sum_{j=-\infty}^{{\rm log}_2 (2N^{-1})} 
%\| W_1^N T_{N,j} W_2^N \|_{\mathcal{C}^\alpha(L^2(\mathbb{T}^{d+1}))} \\
\lesssim &
N^{(d\mu - 2(\frac14-\frac1\rho))\frac 2\alpha} N^{-(1-\frac d\alpha(1-\theta) -\frac1\alpha)} \| W_1 \|_{L^{\frac{\rho\alpha}{2}}_tL^\alpha_x(\mathbb{T}^{d+1})} \| W_2 \|_{L^{\frac{\rho\alpha}{2}}_tL^\alpha_x(\mathbb{T}^{d+1})}. 
\end{align*}
The parameter $\rho\geq4$ is determined to establish the scaling condition $ 2\cdot\frac{2}{\rho\alpha} + \frac d\alpha =1$ which means $\frac1\rho = \frac{\alpha - d}{4}$. 
From this and a few computations we learn $\alpha$ is restriced to $d\leq\alpha\leq d+1$.  %from the restriction $ 0\leq \frac1\rho\leq \frac14$, $\frac1\rho = \frac{\alpha - d}{4}$ implies $d\leq \alpha \leq d+1$.
%Inserting $\frac1\rho = \frac{\alpha - d}{4}$ to the power of $N$, 
Then we finally have 
\[
\big\| W_1^N \mathcal{E}_N \mathcal{E}_N^* W_2^N \big\|_{\mathcal{C}^\alpha(L^2(\mathbb{T}^{d+1}))}  
\lesssim 
N^{\frac d\alpha \mu} \| W_1 \|_{\beta,\alpha} \| W_2 \|_{\beta,\alpha}, 
\]
for any $\alpha \in [d,d+1]$, $\mu \in (\frac{d+1-\alpha}{d},1]$ and $\frac 2\beta +\frac d\alpha =1$. 
In particular, taking $\mu = \frac{d+1-\alpha}{d} + \varepsilon$, we arrive at \eqref{e:dual-beyond}. 
\end{proof}

\begin{acknowledgements}
This work was supported by Grant-in-Aid for JSPS Research Fellow no. 17J01766. This work grows out the collaboration with Professors Neal Bez, Younghun Hong, Sanghyuk Lee and Yoshihiro Sawano \cite{BHLNS}. The author thank Neal Bez for introducing me to this problem, Sanghyuk Lee for sharing very useful insight, Younghun Hong for giving me many comments from view point of PDE perspective and Yoshihiro Sawano for making the paper nicer. 
\end{acknowledgements}

\end{document}